\documentclass[11pt]{siamart171218}

    \usepackage{geometry}
\usepackage[absolute,overlay]{textpos}

\usepackage{graphicx}
\usepackage{subfig}
\usepackage{tabularx}
\usepackage{wasysym,mathtools}

\usepackage{mathrsfs}     
\usepackage{helvet}         
\usepackage{courier}        
\usepackage{type1cm}      
\usepackage{color}
\usepackage{url}

\usepackage{geometry,calc,color}
\usepackage{amsmath,amssymb}

\usepackage{enumerate}
\usepackage{arydshln}
\usepackage{siunitx}
\usepackage{booktabs}
\usepackage{multirow}
\usepackage{tikz}
\usepackage{pgfplots}
\usepackage{enumitem}
\pgfplotsset{compat=1.9}

\usepackage{textcomp}

\newcommand{\extension}[1]{\widehat{#1}}
\newcommand{\hH}{\left(\frac{h}{H}\right)}

\newcommand{\FH}{\mathcal{F}_H}

\newcommand{\Wl}{\mathcal{W}^\ell}
\newcommand{\Wk}{\mathcal{W}^k}

\newcommand{\cW}{\mathcal{W}}
\newcommand{\NE}{\mathcal{N}_E}
\newcommand{\NF}{\mathcal{N}_F}

\newcommand{\WF}{\mathcal{W}_F}
\newcommand{\YF}{\mathcal{Y}_F}

\newcommand{\Fl}{\mathcal{F}^\ell_H} 

\newcommand{\x}{x}

\newcommand{\normf}{\mathbf{n}_\face}
\newcommand{\normK}{\mathbf{n}_K}

\newcommand{\labeldofv}{$\face$.1)}
\newcommand{\labeldofe}{$\face$.2)}
\newcommand{\labeldoff}{$\face$.3)}
\newcommand{\labeldofK}{K)}
\newcommand{\canonical}{\varphi}

\newcommand{\wirebasket}{\Sigma}
\newcommand{\xpoints}{\Xi}
\newcommand{\vertex}{V}
\newcommand{\NV}{\mathcal{N}_\vertex}

\newcommand{\BasePoly}{\mathfrak{B}^\face}
\newcommand{\pjf}{p_j^\face}
\newcommand{\hBasePoly}{\widehat{\mathfrak{B}}}
\newcommand{\Poly}{\mathbb{P}}
\newcommand{\Lproj}{\Pi^0}

\newcommand{\tIf}{I_{\p-2}^\face}
\newcommand{\Harm}{\mathcal{H}}

\newcommand{\dof}[2]{\mathfrak{y}_#2(#1)}
\newcommand{\nodei}{y_i}
\newcommand{\dofi}{\mathfrak{y}_i}
\newcommand{\support}{\mathop{supp}}

\newcommand{\DeltaF}{\Delta^{\mathcal{F}}}
\newcommand{\DeltaV}{\Delta^{\mathcal{X}}}

\newcommand{\DeltaE}{\Delta^{\mathcal{E}}}

\newcommand{\CjumpE}{\tau_{E}}
\newcommand{\CjumpF}{\tau_{F}}
\newcommand{\CjumpEF}{\tau_{EF}}

\newcommand{\Nstar}{N_\star}
\newcommand{\NK}{N_K}
\newcommand{\gammastar}{\gamma_\star}
\newcommand{\gammaK}{\gamma_K}
\newcommand{\xf}{x_f}
\newcommand{\xK}{x_K}

\newcommand{\nref}[1]{(\ref{#1})}

\newcommand{\Th}{\mathcal{T}_h}
\newcommand{\Thk}{\mathcal{T}_h^\ell}

\newcommand{\EH}{\mathcal{E}_H}

\newcommand{\Yl}{\mathcal{Y}^\ell}
\newcommand{\Xl}{\mathcal{X}^\ell}

\newcommand{\cX}{\mathcal{X}}

\newcommand{\Ni}{\mathcal{N}_i}

\newcommand{\YE}{\mathcal{Y}_E}

\newcommand{\hW}{\widehat W_h}

\newcommand{\tW}{\widetilde W_h}

\newcommand{\hSop}{\mathcal{\widehat S}}
\newcommand{\tSop}{\mathcal{\widetilde S}}

\newcommand{\Rop}{\mathcal{R}}
\newcommand{\Bop}{\mathcal{B}}

\newcommand{\Lifthl}{\mathcal{L}_h^\ell}
\newcommand{\Lifth}{\mathcal{L}_h}

\renewcommand{\vec}[1]{\mathbf{#1}}

\newcommand{\p}{m}
\newcommand{\Bone}{{\mathbb{B}_\p}}
\newcommand{\Vfp}{V^{\face}_\p}
\newcommand{\VKp}{V^{K}_\p}
\newcommand{\tVfp}{\widetilde{V}^\face_\p}
\newcommand{\dimV}{N}

\newcommand{\Svem}{S_a^K}

\newcommand{\ED}{\mathcal E_D}

\newcommand{\FETI}{\mathcal{M}}
\newcommand{\IdtW}{\mathbf{1}_{\tW}}

\newcommand{\oml}{\Omega^\ell}

\newcommand{\normWh}[1]{\| #1 \|_{1/2,*}}
\newcommand{\snormWh}[1]{| #1 |_{1/2,*}}

\newcommand{\face}{f}

\newcommand{\Pinabla}{\Pi^\nabla_K}
\newcommand{\Pinablaf}{\Pi^\nabla_\face}

\newcommand{\Pscottzhang}{\Pi_{SZ}}

\ifpdf
  \DeclareGraphicsExtensions{.eps,.pdf,.png,.jpg}
\else
  \DeclareGraphicsExtensions{.eps}
\fi

\newtheorem{problem}{Problem}[section]
\newtheorem{assumption}{Assumption}[section]

\theoremstyle{remark}
\newtheorem{remark}[theorem]{Remark}
\numberwithin{equation}{section}

\newcommand{\roundPrecision}{2}

\headers{FETI-DP for 3D VEM}{S. Bertoluzza, M. Pennacchio, and D. Prada}

\title{FETI-DP for the three-dimensional Virtual Element Method\thanks{
	Submitted December 17, 2018, revised February 27, 2020.
\funding{The work of the authors was realized in the framework of ERC Project CHANGE, which received funding from the European Research Council (ERC) under the European Union's Horizon 2020 research and innovation program grant 694515.
}}}

\author{Silvia Bertoluzza\thanks{Istituto di Matematica Applicata e Tecnologie Informatiche ``E. Magenes", IMATI-CNR, Pavia, PV, 27100, Italy 
  (\email{silvia.bertoluzza@imati.cnr.it}, \email{micol.pennacchio@imati.cnr.it}, \email{daniele.prada@imati.cnr.it} }
\and Micol Pennacchio\footnotemark[2]
\and Daniele Prada\footnotemark[2]
}

\usepackage{amsopn}

\newcommand{\rosa}[1]{{\color{black} #1}} 
\newcommand{\rosso}[1]{{\color{black} #1}} 

\begin{document}

\maketitle

\begin{textblock*}{14cm}(3.5cm,24.8cm) 
	\noindent
	\scriptsize{This is the final draft of the paper published in {\it SIAM J. Numer. Anal.} Vol. 58, No. 3, pp. 1556-1591 (2020),  doi: {10.1137/18M1233303}. The original publication is available at https://epubs.siam.org.
	}
\end{textblock*}
\begin{abstract} 
	We deal with the finite element tearing and interconnecting dual primal (FETI-DP)  preconditioner  for  elliptic problems discretized by the virtual element method (VEM).  
	We extend the result of \cite{FETI_VEM_2D} to the three dimensional case. We prove polylogarithmic condition number bounds, independent of the number of subdomains, the mesh size, and jumps in the diffusion coefficients. 
	Numerical experiments validate the theory. 
\end{abstract}




\section{Introduction}

Methods for the solution of PDEs based on polytopal meshes have recently attracted an increasing attention, mainly due to the necessity of tackling what is nowadays a bottleneck in the overall process of simulating real life phenomena, namely the task of mesh generation.
Several methods have been recently introduced which allow for quite general polygonal or polyhedral elements, such as 
mimetic finite differences \cite{book_mimetic,brezzi_mimetic}, discontinuous Galerkin-finite element method (DG-FEM) \cite{Rev_DG,Cangiani_hpDGVEM}, hybridizable and hybrid high-order methods
\cite{Cockburn_LDG,Ern_LDG}, weak galerkin method  \cite{Weak_FEM}, BEM-based FEM \cite{BEM_FEM}  and polygonal FEM \cite{pol_FEM} to name a few.

Here we deal with the virtual element method (VEM) \cite{basicVEM}, a discretization technique that can be
considered as an extension of the FEM 
 to polytopal tessellations. 
In such a method, 
local approximation spaces containing polynomial functions are defined and assembled in a global conforming approximation space,  but the explicit construction and integration of the associated shape functions are avoided, whence the name {\em virtual} \cite{basicVEM}.
The evaluation of the operators and matrices needed in the implementation of the method is  carried out by relying only on an implicit knowledge of the local shape functions,  as described in \cite{hitchVEM} (see also \cite{antonietti_p_VEM,beirao_hp,beirao_hp_exponential}, where the $p$ and $hp$ versions of the method are discussed and analyzed). 
Though introduced fairly recently, such a method 
 has already been applied and extended to  a wide variety of different model problems;
we recall applications to: parabolic problems \cite{beirao_parab}, Cahn-Hilliard, Stokes, Navier-Stokes and Helmholtz equations \cite{Antonietti_VEM_Stokes,Antonietti_VEM_Cahn,beirao_stokes,beirao_Navier_Stokes,perugia_Helmholtz}, linear and nonlinear elasticity problems \cite{beirao_elastic,beirao_linear_elasticity,VEM_3D_elasticity}, general elliptic problems in mixed form \cite{VEM_mixed}, fracture networks \cite{VEM_discrete_fracture}, Laplace-Beltrami equation \cite{VEM_Laplace_Beltrami}.

In this paper, we focus on the linear system of equations associated with the VEM discretization.
As happens in the case of finite elements, the  efficient solution of such a linear system is of paramount importance to fully exploit the potential of the method. Little work has been done on this issue up to now, all limited to the spatial dimension two. 
 	First works in the literature tackled the increase of 
the condition number appearing already at the level of the elementary stiffness matrix,  
due either to a degradation of the quality of the tessellation and/or to the increase in the polynomial order  of the method \cite{beirao_hp,dassi_mascotto_3DVEM,Mascotto_illcond,berrone_borio_17}.
If we consider rather  the increase of the condition number resulting from refining the discretization,  to the best of our knowledge the approaches considered up to now are  domain decomposition (\cite{Calvo_Schwarz,Calvo_MG,FETI_VEM_2D,Prada_enumath_17}) and multigrid (\cite{antonietti_p_VEM}, for $p$ refinement). 
In the present paper we extend to the three dimensional case the results obtained  in \cite{FETI_VEM_2D}. More precisely, 
 we focus on one of the most efficient preconditioning techniques: the dual-primal finite element tearing and interconnecting (FETI-DP) \rosso{\cite{Farhat:partI,Toselli.Widlund}}, a non overlapping domain decomposition method where the problem is reformulated as a constrained optimization problem and solved by iterating on the set of Lagrange multipliers representing the fluxes across the interface between the non overlapping subdomains.
The FETI-DP method has been already extensively studied in the context of many different discretization methods --  spectral elements  \cite{Pavarino2007.BDDC.FETIDP,klawonn2008spectral}, mortar discretizations \cite{KIM.mortar.FETIDP}, 
NURBS discretizations in isogeometric analysis \cite{Pavarino_iso_FETIDP}. 

Following the approach presented in \cite{FETI_VEM_2D} for two dimensional domains, which mainly relies on the properties of the trace of the discrete space on the interface of the domain decomposition, we prove that the properties of scalability, quasi-optimality and independence on the discontinuities of the elliptic operator coefficients across subdomain interfaces, that are known for the finite element case, still hold when dealing with VEM. 
More specifically, we show that the condition number of the preconditioned matrix is bounded by a constant times
the factor $
(1+\log(H/h))^2,
$
where $H$ and $h$ are, respectively, mesh-size of the subdomain decomposition and of the tessellation, see Theorem \ref{bound_FETI-DP}.
In order to do so, we need to prove several inequalities related to the VEM approximation space, by only relying on the implicit definition of the discrete functions, which, we recall, are not explicitly known.

We observe that, since we are in the framework of \cite{MandelSousedik},
the equivalence of the BDDC (balancing domain decomposition by constraint) and the FETI-DP preconditioners holds.
Therefore the bound for the condition number obtained here also yields an estimate on the BDDC preconditioner for VEM.

The paper is organized as follows. The basic notation,
functional setting and the description of the VEM 
are given
in Section~\ref{sec:VEM}. 
The dual-primal preconditioner is introduced and analyzed in Section~\ref{sec:dualprimal}, 
whereas some relevant properties of the virtual element (VE) discretization space, mainly used for the proof, are presented in 
Section~\ref{sec:VEMprop}. 
The analysis of the preconditioner, with the proof of the estimate for the condition number (Theorem \ref{bound_FETI-DP}), is carried out in Section \ref{sec:FETI_prec} where we also give some detail specific to its implementation in the VEM framework.
Numerical experiments that validate the theory are presented in Section~\ref{sec:numerical}.


\newcommand{\Vf}{\mathcal{V}^\face}
\newcommand{\VK}{\mathcal{V}^K}
\newcommand{\hatpjf}{{\widehat p}_j^f}

\section{The VEM}\label{sec:VEM}
We start by recalling the definition and the main properties of the virtual element method
\cite{3DVEM}. 
To fix the ideas, we focus on the following elliptic model  problem: 
\begin{equation}\label{strong}
-\nabla \cdot(\rho \nabla u) = g \ \text{ in } \Omega, \qquad u = 0 \ \text{ on } \partial \Omega,
\end{equation}
with $g \in L^2(\Omega)$,
where $\Omega \subset \mathbb{R}^3$ is (for simplicity) a convex polyhedron. We assume that the coefficient $\rho(x)$ is a scalar function of $x$ such that for almost all $x \in \Omega$,  $\alpha \leq \rho(x) \leq M$ for two constants $M \geq \alpha >0$.

\subsubsection*{The tessellation} The virtual element method looks for an approximation to the solution of \eqref{strong} in a conforming subspace of $H^1_0(\Omega)$ constructed on a polyhedral tessellation of $\Omega$.
Let us  then start by introducing the assumptions on the tessellation. We consider a family $\{ \Th \}_h$ of tessellations of $\Omega$ into a finite number of  polyhedra $K$, which we assume to be shape regular according to the following definition (quite standard in the theoretical study of VEM).
\begin{definition}\label{shape_regular} We say that a polyhedron $K$ is shape regular of diameter $h_K$ with constants $\NK>0$ and $\gamma_K >0$ if $K$  satisfies the following assumptions (\cite{Lipnikov}): 
	\begin{enumerate}
		\item $K$ has at most $\NK$ faces and $\NK$ edges;
				\item for   every face $\face$ and every edge $e$ we have: 
		\[
		\quad \gammaK h_K^3 \leq | K |, \qquad \gammaK h_K^2 \leq | \face |, \qquad \gammaK h_K \leq | e |;
		\]
		\item for each face $\face$, there exists a point $\xf \in \face$ such that $\face$ is star-shaped with respect to every point in the disk of radius $\gammaK h_K$ centered at $\xf$;
		\item there exists a point $\xK$ such that $K$ is star-shaped with respect to every point in the sphere of radius $\gammaK h_K$ centered at $\xK$;		
		\item   for every face $\face$ of $K$,  there  exists a pyramid  contained in $K$ such that its base equals to $\face$, its height equals  $\gammaK h_K$ and the projection of its vertex onto $\face$ is $\xf$.
	\end{enumerate}
\end{definition}
{We recall (\cite{Lipnikov}) that  $K$ is shape regular if and only if it admits a conformal decomposition  that is made of less than $\NK$ shape regular tetrahedra and includes all its vertices.
}

\

We assume the tessellation $\Th$ to be quasi-uniform and uniformly shape regular. More precisely we make the following assumption.
\begin{assumption}\label{ass:tessellation} We assume that there exist two constants $\Nstar$ and $\gammastar$ such that the tessellation $\Th$ verifies the following assumptions:
	\begin{enumerate}
	
		 		\item {$\Th$ is geometrically conforming, that is for each $K,K'\in \Th$,  if a vertex,  edge  or  face of $K$ is contained in $K \cap K'$,  it is also, respectively, a vertex,  edge, or face of $K'$;} 	
		\item all $K \in \Th $ are shape regular of diameter $h_K$ with constants $\gammaK \geq \gammastar$ and $\NK \leq \Nstar$; 
		\item the tessellation is quasi uniform, that is there exists an $h$ such that for all $K \in \Th $, $h_K \simeq h$.
			\end{enumerate}
	
\end{assumption}

As we are interested here in explicitly studying the dependence of the estimates that we are going to
prove on the number and size of the subdomains and the number and size of the elements of the tessellations, throughout the paper we will employ the notation
$A \lesssim B$ (resp. $A \gtrsim B$) to say that the quantity $A$ is bounded from above (resp. from below)
by $cB$, with a constant $c$ independent of the diffusion coefficient $\rho$ in the PDE (and in particular, independent of $M$, $\alpha$, and of their jump across the interface of the decomposition), that can depend on the polynomial order $\p$ of the VEM method, and depending on the tessellation and on the domain decomposition only via the  constants in  Assumptions \ref{ass:tessellation}  and \ref{ass:macroel}.   
The expression $A \simeq B$ will stand for $ A \lesssim B \lesssim A$.

\subsubsection*{The local face space}  The order $\p$  virtual element discretization space on $\Th$ is
defined element by element starting from the edges of the tessellation, where the discrete functions are defined as degree $\p$ polynomials, then subsequently defining it on the faces and in the interior of the polyhedra. As the space on the faces, and the degrees of freedom  identifying the elements of such a space, will play a key role in the definition and analysis of the FETI preconditioner, it is worth  recalling  its definition in some detail.

 On the boundary of each face  $f$ we  introduce the space: 
\begin{gather*}
\Bone(\partial \face) = 
\left\{g\in C^0(\partial \face) :  g_{|{e}}\in \mathbb{P}_\p(e) \text{ for all edge } e \subseteq \partial \face 
\right\},
\end{gather*}
where, for any one, two or three dimensional domain $D$, $\mathbb{P}_\p(D)$ denotes the set of order  $\p$ polynomials on $D$. 
 In order to define the discrete face space $V^\face$, $\face$ being a face of one of the polyhedra of the tessellation, we start by introducing an auxiliary space  $\tVfp$, defined   as 
\begin{gather}\label{defspace1}
\tVfp =  \{ v \in C^0(f):\, v|_{
	\partial \face} \in\Bone(\partial \face),\, 
\Delta v \in \mathbb{P}_\p(\face) \}.
\end{gather}
It is known \cite{basicVEM} that an element in $\tVfp$ is uniquely identified once its trace on $\partial \face$ and its moments up to order $\p$ (or, equivalently, the  scalar products with the elements of a basis $\BasePoly_\p$ for the space $\Poly_\p(f)$) are known. 

\

The local face space $\Vfp$ is a suitable subspace of $\tVfp$. In order to define it, we introduce the projection $\Pinablaf: H^1(\face) \to \Poly_\p(\face)$, orthogonal to the scalar product
\begin{equation}\label{scalar_prod}
(v,w)_{1,f} = \int_\face \nabla v(x) \cdot \nabla w(x) + \bar v^\face \bar w^\face, \qquad\text{where, for all $w\in V_h$},\ \bar w^\face  = \int_{\partial \face} w.
\end{equation}
To define $\Vfp$ we choose two bases $\BasePoly_{\p - 2} \subset \BasePoly_{\p}$ for the two embedded subspaces $\Poly_{\p-2}(\face) \subset \Poly_\p(\face)$ of polynomials of order less than or equal to, respectively, $\p-2$ and $\p$ (where, for $\p=1$ we write conventionally $\Poly_{-1}(\face)=\{0\}$):
\[
\BasePoly_{\p-2} = \{ \pjf ,\ j=1,\cdots,N_{\p-2}\}, \qquad \BasePoly_\p = \BasePoly_{\p-2} \cup \{ \pjf , \ j=N_{\p-2}+1, \cdots, N_\p\},
\]
where, for $\p  \geq -1$, $N_\p$ denotes the dimension of the space $\Poly_\p(\face)$. We assume that $\BasePoly_\p$ is a suitably scaled Riesz basis. More precisely, we make the following assumption.
\begin{assumption}\label{assumption:Rieszbasis}
For all polynomial $q \in \Poly_\p$ it holds that
\begin{equation}\label{polyriesz1}
\int_\face | q |^2_{0,\face} \simeq h^2 \sum_{j=1}^{N_\p} | \int q \pjf  |^2.
\end{equation}
\end{assumption}

We then define $\Vfp$ as
\begin{gather}\label{defspacek}
\Vfp =  \{ v \in \tVfp:\  \int_\face   v \pjf  = \int_\face  (\Pinablaf v) \pjf ,\  j = N_{\p-2} +1, \cdots, N_\p\}.
\end{gather}
We immediately see that, as $\Pinablaf$ is a projector, for all polynomials $q \in \Poly_\p$, $\Pinablaf q = q$, and hence, as $\Delta q \in \Poly_{\p-2}(\face) \subset \Poly_\p(\face)$ and $q|_{\partial\face} \in \Bone(\face)$, it holds that $\Poly_\p(\face) \subset \Vfp$. 

It is also easy to see that a function $v$ in $\Vfp$ is uniquely determined by 
\begin{enumerate}
	\item[\labeldofv] the values of $v$ at the vertices of $\face$ (\emph{vertex degrees of freedom});
	\item[\labeldofe] the values of $v$ at the $\p - 1$ interior nodes of the  $\p + 1$ points Gauss-Lobatto integration rule (\emph{edge degrees of freedom});
	\item[\labeldoff] the  value of the scalar products  (\emph{face degrees of freedom}) 
	\[
	\int_\face v \pjf ,\quad \text{for all $\pjf  \in \BasePoly_{\p - 2}$} .
	\]
\end{enumerate}
In fact, the degrees of freedom \labeldofv-\labeldofe\ uniquely determine the trace of $v$ on $\partial \face$, which is an element of $\Bone(\partial \face)$. Moreover, the values of \labeldofv--\labeldoff\ uniquely determine $\Pinablaf v$, as they are sufficient to compute $(v,q)_{1,\face}$ for all $q \in \Poly_\p$.
Indeed \cite{basicVEM}, for $q \in \mathbb{P}_\p(K)$, Green's formula yields
\begin{equation}\label{Greenface}
\int_\face \nabla v(x) \cdot \nabla q(x) \,dx = \int_{\partial \face }v(s) \nabla q(s) \cdot \mathbf{n}_\face(s) \,ds  - \int_\face v(x) \Delta q(x)\,dx,
\end{equation}
where $\normf$ is the outer unit normal to $\face$.
The values of \labeldofv--\labeldoff\ give access to $v$ on $\partial \face$ and to its $L^2(\face)$ scalar product with any polynomial of order up to $\p-2$, thus allowing us to compute the right hand side of \eqref{Greenface}, as $\Delta q \in \Poly_{\p-2}$. In view of the definition of $\Vfp$,
this implies, among other things, that the knowledge of the values of $\face$.1)--$\face$.3) allows us to compute exactly the $L^2(\face)$ scalar product of a function in $\Vfp$ times any polynomial of order less than or equal to $\p$.

\

 Observe that the space $\Vfp$ does not depend on the choice of the basis functions in $\BasePoly_{\p - 2}$, but it does depend on the choice of the basis functions in  $\BasePoly_\p \setminus \BasePoly_{\p - 2}$. Usually, the $\pjf $, $j = N_{\p-2}+1, \cdots, N_{\p}$, are chosen to be the scaled monomials of order greater than $\p-2$ and less than or equal to $\p$, though other choices are possible (and even advisable, for $\p >> 1$).

\

\begin{remark} Requiring that Assumption \ref{assumption:Rieszbasis} holds is equivalent to requiring that for all vectors $(c_j) \in \mathbb{R}^{N_\p}$ it holds that
	\[
	\int_\face | \sum_j c_j \pjf  |^2 \simeq h^{-2} \sum_j | c_j |^2 
	\]
	(see \cite{Christensen}).
	It is a quite natural assumption,  satisfied by many possible choices of $\BasePoly_\p$. Of course, \eqref{polyriesz1} holds for any suitably normalized orthogonal basis for $\Poly_\p(\face)$.
	In \cite{Chen2018}, the authors prove that it is indeed satisfied for the basis of rescaled monomials $m_\alpha$, $\alpha = (\alpha_1,\alpha_2) \in \mathbb{N}^2$, $\alpha_1+\alpha_2 \leq \p$, where we set
	\[
	m_\alpha(\x) = h^{-2}
	\left(
	\frac{ \x - \x_f }
	{h}	
	\right)^\alpha \quad \text {with, for $x = (x_1,x_2)$, $\alpha = (\alpha_1,\alpha_2)$,}\quad \x^\alpha = x_1^{\alpha_1} x_2^{\alpha_2}.
	\]
	A similar argument shows that such a property is satisfied by the rescaled version of any basis $\hBasePoly = \{
	 \hatpjf , \ j = 1,\cdots,N_\p\}$  for $\mathbb{P}_\p$
	\[
	\pjf  (x) = h^{-2} \hatpjf \left(
	\frac{ \x - \x_f }
	{h}
	\right).
	\]
\end{remark}

\begin{remark} 
	
	For the sake of simplicity, the basis functions $\pjf $ are assumed to be normalized in such a way that \eqref{polyriesz1} holds rather than the maybe more natural
\begin{equation}
\int_\face | q |^2_{0,\face} \simeq \sum_{j=1}^{N_\p} | \int q \pjf  |^2 
\end{equation}
(that would hold, for instance, in the case of orthonormal bases).
Remark, however, that the normalization of the degrees of freedom is transparent to the action of the preconditioner, so that both the design of the preconditioner itself and the theoretical bounds on the resulting condition number turn out not to depend on it. \end{remark}



\subsubsection*{The element space and the global space} In order to build the discrete space on a polyhedron  $K$, we assemble the local face spaces $\Vfp$ to build a local space defined on the boundary of $K$:
$$ \Bone(\partial K)  = 
\left\{g\in C^0(\partial K) : g_{|f}\in \Vfp \text{ for all face } \face \subseteq \partial K 
\right\}.
$$
The local space on $K$ is finally defined as 
\begin{equation}\label{defspace2} \VKp: = 
\left\{v\in C^0(K) :  v_{|\partial K}\in \Bone(\partial K)  \text{ and } \Delta v\in \mathbb{P}_{\p-2}(K)
\right\}.
\end{equation}
As $\Bone(\partial K)$ contains the subspace of continuous piecewise polynomials of order up to $\p$, it is immediate to see that $\VKp$ contains the polynomials of order up to $\p$. 

\

A function in $\VKp$ is uniquely determined by the value of the vertex, edge and face degrees of freedom $\face$.1) -- $\face$.3) corresponding to all the faces of $K$, plus
\medskip
\begin{enumerate}
	\item[($K$)] the values of the moments of $g$ in $K$, up to order $\p-2$ (\emph{interior degrees of freedom}).
	\end{enumerate}
\medskip
Such degrees of freedom are sufficient to compute the action of the projection $\Pinabla: V_h \to \Poly_\p(K)$, orthogonal to the scalar product
\[
(w,v)_{1,K} = \int_K \nabla w \cdot \nabla v + \bar w^K \bar v^K, \qquad \text{ where } \bar w = \int_{\partial K} w. 
\]
In fact, for $q \in \Poly_\p$, integrating by parts we obtain
\begin{equation}\label{Green3D}
\int_K \nabla w \cdot \nabla q = - \int_K w \Delta q + \int_{\partial K} w \nabla q\cdot \normK,
\end{equation}
where $\normK$ is the outer unit normal to $K$. The right hand side can be computed since, on the one hand, $\Delta q$ is in $\Poly_{\p-2}(K)$ and, on the other hand, since $\nabla q \cdot \normK$ is in $\Poly_{\p-1}(\face)$ for all faces $\face \subset \partial K$,  as observed in the previous section, the degrees of freedom $\face.1$--$\face.3$ enable us to compute the second integral on the right hand side.

\

We finally define the global virtual element space on $\Omega$ by gluing the local spaces:
$$ V_h  = 
\left\{V \in C^0(\Omega) : v_{|K}\in \VKp\text{ for all } K  \in \Th
\right\}.
$$

\begin{remark}
	The definition of $\Vfp$ can appear quite cumbersome, and one might think that there are simpler options for defining the space $\Vfp$. Two possibilities come to mind. The first one is to define $\Vfp = \widehat V^\face_{\p}$ with
	\[
	\widehat V^\face_{\p} = \{
	v \in C^0(f):\, v|_{
		\partial \face} \in\Bone(\partial \face),\, 
	\Delta v \in \mathbb{P}_{\p - 2}(\face)	
	\},
	\]
	as is done in the definition of the simplest form of the two dimensional virtual element method. However, if we define $\Vfp$ in this way, the knowledge of the values of the degrees of freedom \labeldofv--\labeldoff\ would not be sufficient to compute the $L^2(\face)$ scalar product of a function in $\Vfp$ with an order $\p - 1$ polynomial. This is needed to evaluate the projector $\Pinabla$, which, as we will see, is necessary to build the discretization of the PDE 
	that we want to solve. The second possibility would be to let $\Vfp = \tVfp$. This choice would indeed allow the computation of $\Pinabla$ but it would lead to a dramatic increase in the number of degrees of freedom (an extra $N_\p - N_{\p-2}$ per face), without any relevant increase in the order of approximation.
\end{remark}

\subsubsection*{The global degrees of freedom} 
It is immediate to see that a function in $V_h$ is uniquely determined by the values of the degrees of freedom \labeldofv--\labeldoff~ and \labeldofK~ for all vertices, edges, faces, and polyhedra of the tessellation. Requiring the continuity of $v_h$ reduces to asking that all \labeldofv, \labeldofe\  and \labeldoff\ are single valued, that is, that they produce the same result when evaluated on the restriction of $v_h$ to two polyhedra sharing, respectively, a vertex, edge or face.

\

It will be convenient, in the following, to explicitly introduce  functionals $\dofi: V_h \to \mathbb{R}$, $i = 1,\cdots,\dimV$ with $\dimV$ denoting the dimension of the global space $V_h$, mapping the elements of $V_h$ to the value of the corresponding $i$-th degree of freedom. In other words,  
{a function $w_h \in V_h$ is uniquely determined by the vector  $(\dofi(w_h))_{i=1}^N$ where $\dofi: V_h \to \mathbb{R}$ is either one of the following:
	\begin{enumerate}
		\item $\dofi(w_h)$ is the value of $w_h$ at one of the vertices of the tessellation $\Th$ (vertex degrees of freedom);
		\item $\dofi(w_h)$ is the value of $w_h$ at one of the $\p - 1$ interior nodes of the  $\p + 1$ points Gauss-Lobatto quadrature formula (edge degrees of freedom);
		\item $\dofi(w_h)$ is the scalar product $\int_\face w_h \pjf $, $\pjf  \in \BasePoly_{\p -2}(\face)$, with $\face$ being one of the faces of the tessellation (face degrees of freedom);
		\item $\dofi(w_h)$ is one of the moments of order up to $\p-2$ of $w_h$ over one of the elements $K$ (interior degrees of freedom).
	\end{enumerate}	
	Note that edges and vertex degrees  of freedom are ``nodal'', that is they correspond to a node, that we will denote $\nodei$.
	
}

\

\newcommand{\annihilator}{\mathcal{Z}}
\newcommand{\subsetomega}{\omega}

{
We let $\Upsilon$ denote the set of all the degrees  of freedom of $V_h$,
\[\Upsilon = \{ \dofi,\ i=1,\cdots,N \} = \text{ set of degrees of freedom of $V_h$}. \]  
	We recall that for $\dofi \in \Upsilon$  we can define its support as follows \cite{book_distributions}:
	\[
\support(\dofi) = \Omega \setminus\mathcal{Z}(\dofi),
	\]
	where the \emph{annihilator} $\annihilator(\dofi)$ is the largest open subset of $\Omega$ where $\dofi$ vanishes (where, by definition, a distribution $\dofi$ vanishes on an open subset   $\subsetomega$ of $\Omega$, 
	if and only if $\dofi(w) = 0$ for all $w \in C^{\infty}_0(\subsetomega)$).
	It is easy to see that the support of a vertex or edge degree of freedom is the corresponding node, whereas the support of a face degree of freedom is the closure of the corresponding face, and the support of an interior degree of freedom is the closure of the corresponding element. Observe that 
\[
w_h = 0 \ \text{ on } \support(\dofi) \quad \text{ implies } \quad \dofi(w_h) = 0.
\]

	\begin{remark}\label{remarkintdofs}
		Also for the interior degrees of freedom $K$) we could consider a basis for the space $\Poly_{\p-2}(K)$ different from the monomial basis. For $\p$ increasing this can help to increase the stability of the method (and in particular it allows better constants in the stabilization bound  \eqref{2}). Remark, however, that the choice of such degrees of freedom does not directly enter in the definition of the preconditioner, as they are locally eliminated as a preliminary step, and it affects the theoretical analysis only through such constants. 
	\end{remark}
	
	\begin{remark} It is not difficult to realize that the value of a function  $w_h \in V_h$ at any vertex, edge or face of the tessellation is uniquely determined by the values $\dofi(w_h)$ for $\dofi$ supported, respectively, in the vertex, in the closure of the edge and in the closure of the face. This is trivial for the values at the vertices, but is easily checked also for the value at the edges (a polynomial of degree $\p$ is, of course uniquely determined by its value at nodes of the $\p+1$ point Gauss-Lobatto integration rule) and on the faces (uniquely determined by the  values of \labeldofv--\labeldoff, as observed previously when introducing the local face space).
		Consequently, if $D \subset \Omega$ is obtained as the union of some vertices, edges and/or faces, $w_h|_D$ is uniquely determined by those degrees of freedom $\dofi(w_h)$ such that $\support(\dofi)\subseteq \bar D$. 
\end{remark}
}

\

\subsection*{The VE Discretization}
The Virtual Element Method stems from a Galerkin discretization of problem \eqref{strong}, starting from its variational formulation, which reads
\begin{equation}\label{variational}
\left\{
\begin{array}{l}
\text{find $u \in V:=H^1_0(\Omega)$ such that}\\[1mm]
a(u,v) = (g, v ) \quad \forall v \in V 
\end{array}\right.
\end{equation}
with
\[
a(u,v) = \int_{\Omega}\rho(x) \nabla u(x)\cdot \nabla v(x)\,dx, \qquad (g,v) = \int_{\Omega}g(x) v(x)\,dx.
\]

\

As the functions in $V_h$  are not known in closed form,  it is not possible to directly evaluate the bilinear form $a$ on two of such functions (this would imply solving Poisson  equations on each face and in each element). Clearly we have
\begin{equation}\label{1}
a^{K}(u,v) = a^K(\Pinabla u,\Pinabla v) + a^K(u - \Pinabla u, v - \Pinabla v)
\end{equation}
with $a^{K}(\cdot,\cdot)$ local counterpart of $a(\cdot,\cdot)$. 
The virtual element method is obtained by replacing the operator $a^K$ in the second term  on the right hand side, which cannot be computed exactly, with an ``equivalent'' computable operator $\Svem$. We then define
\[
a^{K}_h(u,v) = a^K(\Pinabla u,\Pinabla v) + \Svem(u - \Pinabla u, v - \Pinabla v),
\]
where $\Svem$ is any continuous symmetric bilinear form satisfying
\begin{equation}\label{2}
a^K(v,v) \simeq \Svem(v,v)\quad \forall v \in \VKp\ \text{ with } \Pinabla v=0.
\end{equation}
Equations (\ref{1}) and (\ref{2}) immediately yield
\begin{equation}\label{elemequiv}
a^K(v,v) \simeq a^K_h(v,v)\quad \forall v \in \VKp\quad \text{ and }\quad  a^K(v,w) = a^K_h(v,w)  \text{ whenever } v  \text{ or } w \in \mathbb{P}_\p(K) .
\end{equation}

Finally, we let $a_h : V_h \times V_h \to \mathbb{R}$ be defined by
\[
a_h(u_h,v_h) = \sum_K a_h^K(u_h,v_h),
\]
and we consider  the following discrete problem: 
\begin{problem}\label{discrete_full} Find $u_h \in V_h$ such that
	\[  
	a_h(u_h,v_h) =  \sum_K \int_K g^K_h v_h
	\qquad  \forall v_h \in V_h,
	\]
	where we let $g_h^K \in \Poly_\p(K)$ denote the $L^2(K)$ projection of $g|_K$ onto the space of polynomials of order less than or equal to $\p$. 
\end{problem}

Different choices of the stabilization bilinear forms $\Svem$ are proposed in the literature. The simplest, widely used choice, is directly expressed in terms of the degrees of freedom as the suitably scaled euclidean scalar product \cite{highorderVEM_3D} 
\begin{equation}\label{stabform1}
\Svem(w,v) = \rho_K h\sum_{i : \support(\dofi) \subseteq \bar K} \dofi(w) \dofi(v).
\end{equation}
An alternative which gives sharper bounds in \eqref{2} as $\p$ increases is the so called ``diagonal recipe'':
\begin{equation}\label{stabform2}
\Svem(w,v) = \sum_{i : \support(\dofi) \subseteq \bar K} \max\{ \rho_K h, a^K(\Pinabla \canonical_i,\Pinabla \canonical_i) \} \dofi(w) \dofi(v) 
\end{equation}
where $\canonical_i \in V_h$ is the canonical basis function corresponding to the degree of freedom $\dofi$, that is the unique function in $V_h$ such that $\dof { \canonical_i } j = \delta_{i,j}$, 
$\delta_{i,j}$ denoting the Kronecker delta. It has also been proposed to replace the sum on the right hand side of \eqref{stabform1} or \eqref{stabform2} by a sum over those $i$ such that $\support(\dofi) \subseteq \partial K$. Other recipes (proposed in two dimensions, but also valid in three dimensions), include suitably scaled versions of the $L^2(\partial K)$ scalar product, and of the bilinear form corresponding to the Laplace-Beltrami operator on $\partial K$ \cite{beirao_stab}.

\

For the study of the convergence, stability and robustness properties of the method we refer to~\cite{basicVEM,beirao_var_coef,Brenner_VEM_estimates}. 

\begin{remark}\label{remarkbarw}
	Different alternatives have been proposed for defining the constant components $\bar w^\face$ and $\bar w^K$ entering the definition of the scalar products $(\cdot,\cdot)_{1,\face}$ and $(\cdot,\cdot)_{1,K}$  and, consequently, of the projectors $\Pinablaf$ and $\Pinabla$. In particular, for $\p = 1$ we can consider
	\[
	\bar w^\face = \sum_{V \text{ vertex of }\face} w(V), \qquad  \bar w^K = \sum_{V \text{ vertex of }K} w(K),
	\]
	while for $\p > 1$ we can set
		\[
		\bar w^\face = \int_\face w, \qquad  \bar w^K = \int_K w.
		\]
		Remark that the choice of $\bar w^K$ does not have any effect on the matrix resulting from  discretizing the bilinear form, as, in its definition, the constant component of $\Pinabla w$ and $\Pinabla v$ are canceled by the $\nabla$ operator.
		On the contrary, the choice of $\bar w^\face$ does have a direct influence on the definition of the $L^2(\face)$ projection of $w$ onto $\Poly_\p(\face)$, which enters the definition of $ \nabla \Pinabla w$ through the second term on the right hand side of \eqref{Green3D}.
	\end{remark}


\section{Some relevant bounds}\label{sec:VEMprop}
In this section we present some bounds that will play a role in the forthcoming analysis. More precisely, letting $K \subseteq  \mathbb{R}^3$ be a shape regular polyhedron of diameter $h$, and letting $\face$ be any face of $K$, we have the following bounds.

\subsection*{Agmon inequality}
For all functions in $H^1(\face)$, it holds that  \cite{beirao_stab,Lipnikov}
\begin{equation}\label{agmon}
\int_{\partial \face} | u |^2 \lesssim  h^{-1} \int_\face | u |^2 + h \int_{\face} | \nabla u |^2. 
\end{equation}

\subsection*{Inverse estimates}  The following two inverse inequalities hold for all $v_h \in \Vfp$ \cite{beirao_stab,Cangianietal17,Chen2018,vacca_Darcy}:
 \begin{eqnarray}\label{inverse1}
 \int_\face | \Delta v_h |^2 &\lesssim& h^{-2} \int_\face | \nabla v_h |^2,\\
 \int_\face | \nabla v_h |^2 &\lesssim& h^{-2} \int_\face | v_h |^2.\label{inverse2}
 \end{eqnarray}

\subsection*{Riesz basis property} 
The following lemma, related to the equivalence between the $L^2(\face)$ norm of a function in $\Vfp$ and the euclidean norm of the vector of its degrees of freedom, holds.
\begin{lemma}\label{lem:Riesz} Let $\face$ be a face of the tessellation. For all $v_h \in \Vfp$ we have that
	\begin{equation}\label{eq:Riesz}
	\int_\face | v_h |^2 \simeq h^2 \sum_{i: \support(\dofi) \subseteq \bar\face} | \dofi(v_h) |^2.	\end{equation}
\end{lemma}

Bound \eqref{eq:Riesz} has been proven in \cite{Chen2018} with a different choice of the edge degrees of freedom (namely, the moments of order up to $k-2$) and for $\BasePoly_{\p}$ being the basis of the rescaled monomials.

 It is not difficult to verify that the proof therein still holds, with some minor changes, for our choice of degrees of freedom, provided Assumption \ref{assumption:Rieszbasis} holds, and provided that for $w \in \Bone(\partial \face)$  we have
\begin{equation}\label{riesz1d}
\int_{\partial\face} | w |^2 \simeq h \sum_{i: \support(\dofi) \subset \partial \face} | \dofi(w) |^2  = h \sum_{i: \support(\dofi) \subset \partial \face} | w(y_i) |^2, 
\end{equation}
where we recall that, for $i: \support(\dofi) \subset \partial \face$, $y_i$ denotes the node corresponding to the degree of freedom, which can be either a vertex of $\face$ or one of the $\p-1$ nodes of the  $\p+1$ points Gauss-Lobatto quadrature formula on an edge of $\face$. 
The norm equivalence \eqref{riesz1d} does indeed hold with our choice of degrees of freedom, and can be proven, using, edge by edge, a scaling argument, and the equivalence of the $L^2$ norm of a polynomial in $\Poly_\p(0,1)$ and the $\ell^2$ norm of the vector of its values at the nodes of the  $\p+1$ points Gauss-Lobatto quadrature formula. 

\

For the sake of completeness let us sketch the main steps of a proof of Lemma \ref{lem:Riesz}, which is slightly different from the one proposed in \cite{Chen2018}, yielding a sharper bound for suitable choices of the face degrees of freedom (such as, for instance, when $\BasePoly_\p$ is chosen to be an $L^2(\face)$ orthogonal basis for $\Poly_\p(\face)$).	

\

Letting $\Lproj_\p: L^2(\face) \to \Poly_\p(\face)$ denote the $L^2(\face)$ orthogonal projection,  and combining \eqref{polyriesz1}  with the definition of orthogonal projection, we can easily see that for all $w \in L^2(\face)$ we have
	\begin{equation}\label{boundl2proj}
	\int_\face | \Lproj_\p w |^2 \simeq  h^2 \sum_{j=1}^{N_\p} | \int_\face \Lproj_{\p} w \pjf  |^2 =  h^2 \sum_{j=1}^{N_\p} | \int_\face w \pjf  |^2.	\end{equation}
		
Then, proving that the right hand side of \eqref{eq:Riesz} is bounded by a constant times the left hand side is not difficult. 	In fact we have that%
		\[
h^2	\sum_{i: \support(\dofi) \subseteq \bar\face} | \dofi(v_h) |^2 =  h^2	\sum_{i: \support(\dofi) \subseteq \partial\face} | \dofi(v_h) |^2 + h^2	\sum_{j = 1}^{N_{\p-2}} | \int_\face w \pjf  |^2.
	\]
	We can bound the first term on the right hand side by combining \eqref{riesz1d}, \eqref{agmon}, and \eqref{inverse2} and the second term 
	by using \eqref{boundl2proj} and the boundedness of the $L^2(\face)$ orthogonal projection.

	\

To prove the converse inequality, as in the proof of \cite{Chen2018}, combining \eqref{riesz1d} with the maximum principle, it is not difficult to prove the following proposition.
\begin{proposition}\label{lem:rieszharmonic} Let $w$ verify $w|_{\partial \face} \in \Bone(\partial \face)$ and $\Delta w = 0$ in $\face$. Then
	\begin{equation}\label{rieszHarm}
\int_\face | w |^2 \simeq h^2 \sum_{i: \support(\dofi) \subset \partial \face} | \dofi(w) |^2.
\end{equation}
\end{proposition}

Additionally, we have the following proposition.
\begin{proposition}\label{lem:rieszinterior} For all $v \in V^0_\face = \{
	v \in H^1_0(\face) : -\Delta v \in \Poly_{\p}(\face)
	\}$ it holds that
	\begin{equation*} \label{rieszint}
	\int_\face | v |^2 \lesssim \int_\face | \Lproj_{\p} v |^2.
	\end{equation*}
\end{proposition}

\begin{proof} Using the Poincar\'e inequality, integrating by parts, applying a Cauchy--Schwartz inequality, and the inverse bounds \eqref{inverse2} and \eqref{inverse1} we obtain
	\begin{gather*} \int_\face | v |^2 \lesssim h^2
 \int_\face | \nabla v |^2 = - h^2 \int_\face v \Delta v = - h^2 \int_\face \Lproj_{\p} v \Delta v
 \lesssim  \sqrt{\int_\face | \Lproj_{\p} v |^2} \sqrt{\int_\face | v |^2}.	\end{gather*}
We conclude by dividing both sides by $\sqrt{\int_\face | v |^2}$.
	\end{proof}

 Letting now $\Harm w$ denote the harmonic lifting of $w|_{\partial \face}$ for all $w \in \tVfp$  we then have
\begin{multline*}
\int_\face | w |^2 \lesssim \int_\face|w - \Harm w|^2 + \int_\face |\Harm w |^2 \lesssim  \int_\face|\Lproj_\p(w - \Harm w)|^2 + \int_\face |\Harm w |^2\\ \lesssim  \int_\face| \Lproj_\p w |^2 + \int_\face |\Harm w |^2
\lesssim  h^2 \sum_{j=1}^{N_\p} | \int_\face w \pjf  |^2 + h^2 \sum_{i: \support(\dofi) \subset \partial \face} | \dofi(w) |^2,
\end{multline*}
where we  used Proposition \ref{lem:rieszinterior}, \eqref{boundl2proj},  and Proposition \ref{lem:rieszharmonic}.
We immediately see that if $w \in \tVfp$ is orthogonal to $\pjf $, $j = N_{\p-2}+1, \cdots, N_\p$, we have that
\begin{equation}\label{boundnonenhanced}
\int_\face | w |^2 \lesssim h^2 \sum_{i: \support(\dofi) \subseteq \bar\face} | \dofi( w ) |^2.
\end{equation}
Let now $\tIf w$ denote the function in $\tVfp$ orthogonal to $\pjf $ for $j = N_{\p-2}+1, \cdots, N_\p$, and verify \[
\Lproj_{\p-2}( \tIf w - w )= 0, \qquad (\tIf w - w)|_{\partial\face} = 0.\] We  easily see that, given any $w \in \tVfp$, we have that $\Pinablaf w = \Pinablaf \tIf w$. Then, for $w \in \Vfp$ we can write
\begin{equation}\label{stellina}
\int_\face | w |^2 \lesssim h^2 \sum_{i: \support(\dofi) \subset \partial \face} | \dofi(w) |^2 + h^2 \sum_{j=1}^{N_{\p-2}} | \int_\face w \pjf  |^2 +  h^2 \sum_{j=N_{\p-2} + 1}^{N_\p} | \int_\face \Pinablaf \tIf w \pjf  |^2  .
\end{equation}
Now, as under our assumptions the operator $\Pinablaf$ is bounded in $H^1(\face)$ uniformly in $h$ \cite{Chen2018}, we have 
\[
 h^2 \sum_{j=N_{\p-2} + 1}^{N_\p} | \int_\face \Pinablaf \tIf w \pjf  |^2  \lesssim \int_\face | \Pinablaf \tIf w |^2 
 \lesssim \int_\face | \tIf w |^2.
\]
Bounding the right hand side by \eqref{boundnonenhanced} and combining with \eqref{stellina} we get the thesis.



\newcommand{\faclogd}{\left( 1+\log\left(\displaystyle\frac{H}{h}\right)\right)}
\newcommand{\logHhd}{\log\left(\displaystyle\frac{H}{h}\right)}

\section{Domain decomposition for the Virtual Element Method}\label{sec:dualprimal}

\subsection{The subdomain decomposition}

We assume that  $\Th$ can be split as $\Th = \cup_\ell \Thk$, inducing a decomposition  of  $\Omega$ as the union of $L$ disjoint polyhedral subdomains $\oml$:
\begin{equation}\label{patch}
\bar \Omega = \cup_\ell \bar \Omega^\ell \quad \text{ with }\quad  {\bar \Omega^\ell}= \cup_{K \in \Thk} K.
\end{equation}
We make the following assumptions.
\begin{assumption}\label{ass:macroel} There exist a constant $\gammastar>0$ and $\Nstar > 0$ such that the subdomain decomposition satisfies the following properties:
	\begin{enumerate}
		\item {it is geometrically conforming, that is, for all $\ell$, if a vertex,  edge, or  face of $\Omega^\ell$ is contained in $\partial\Omega^\ell \cap \partial \Omega^m$, it is also, respectively, a vertex,  edge, or face of $\Omega^m$;}%
		\item the subdomains $\Omega^\ell$ are shape regular (in the sense of Definition \ref{shape_regular}) of diameter $H_\ell$ with constants $\gamma_{\Omega^\ell} > \gamma^\star$ and $N_{\Omega^\ell} < \Nstar$;
%
\item 	for all $\ell$, there exists a scalar $\rho_\ell > 0$ such that
	$\rho|_{\Omega^\ell} \simeq \rho_\ell$; 
	\item 
	the decomposition is quasi uniform:
		there exists an $H$ such that for all $\ell$ we have $H_\ell \simeq H$.
  \end{enumerate}
\end{assumption}

We will refer to the edges and faces of the subdomains $\Omega^\ell$ 
as macro edges and macro faces. We let $\Gamma = \cup \partial\Omega^\ell \setminus \partial\Omega$ denote the skeleton of the decomposition, $\EH$ and $\FH$ denote, respectively, the set of macro edges $E$ and of macro faces $F$  of the subdomain decomposition interior to $\Omega$, and $\FH^\ell$ and $\EH^\ell$ denote the set of, respectively, macro faces and macro edges of the subdomain $\Omega^ \ell$. 

We also let $\wirebasket = \cup_{E\in\EH} \bar E$ denote the {\em wirebasket} (that is, the union of all the edges) of the decomposition. Finally, we let $\xpoints$ denote the set of the vertices of the decomposition (the cross points), see Figure~\ref{fig:split_domain}.

\begin{figure}
	\centering
	\includegraphics[height=3.5cm]{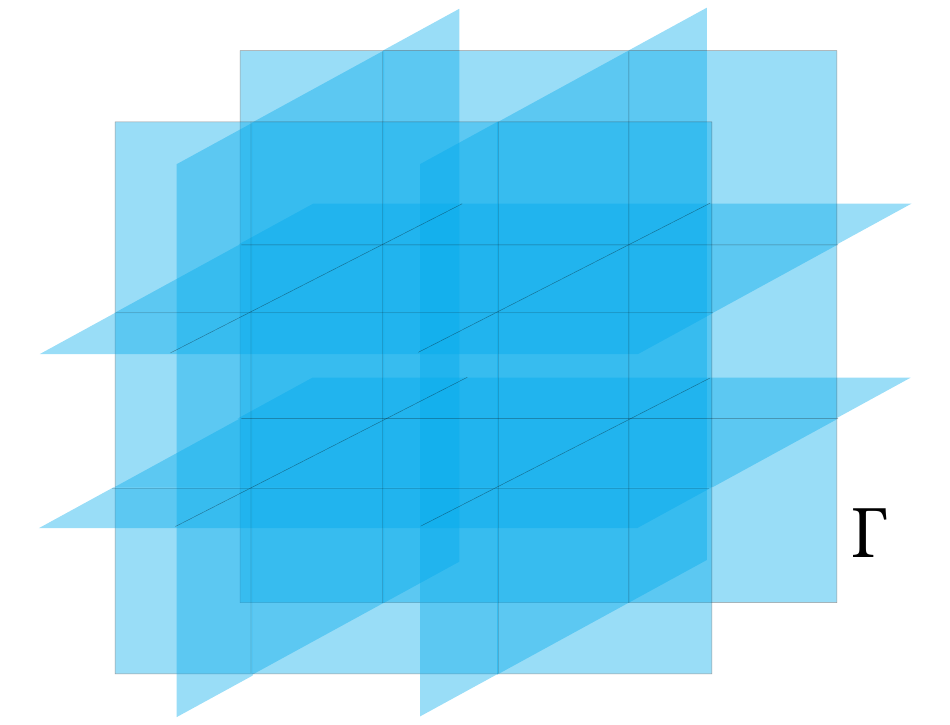}\qquad\qquad
		\includegraphics[height=3.5cm]{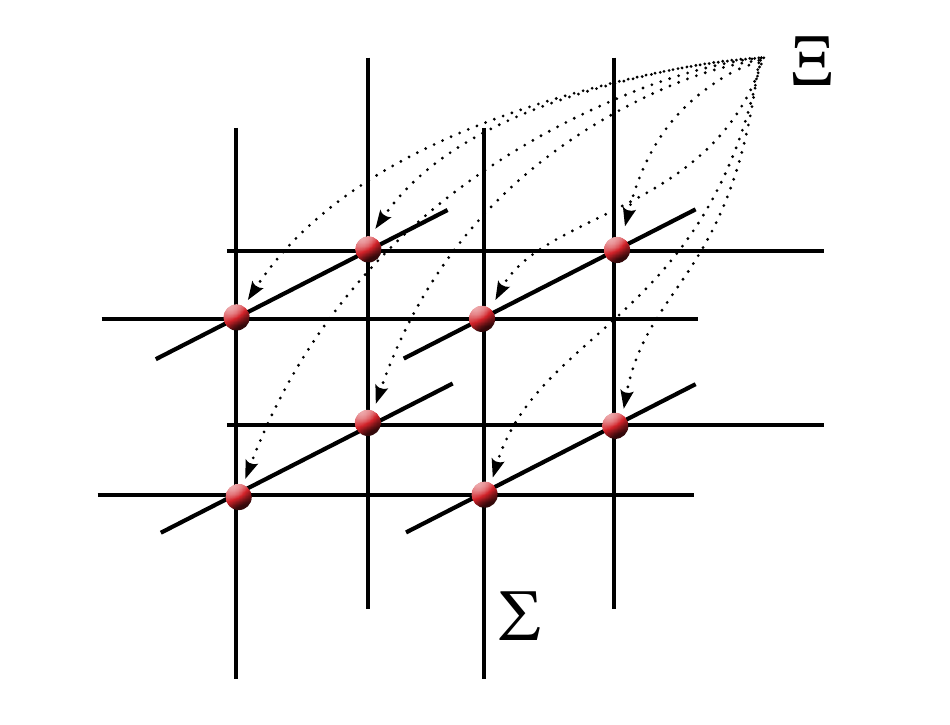}
		\caption{On the left  the interior part of the interface $\Gamma$, on the right the interior part of the corresponding wirebasket $\wirebasket$ and the set of cross points $\xpoints$.}
		\label{fig:split_domain}
\end{figure}
\

{	We would like to point out that Assumption \ref{ass:macroel}, which is quite standard in the framework of domain decomposition methods, is actually also an assumption on the tessellation $\Th$, satisfied, for instance, if $\Th$ is built by first introducing the subdomains $\Omega^\ell$ 
 and then refining them. This is, of course, quite strong an assumption which might leave out some relevant discretization. In this respect, a couple of observations are in order. 
		On the one hand, also in the finite element framework, the literature on FETI methods in irregular domains is quite scarce and, to our knowledge, limited to two dimensions \cite{widlundrough}. On the other hand, while the theory relies on Assumption \ref{ass:macroel}, the design of the preconditioner itself does not, so that in practice, the FETI-DP is also available for subdomains decomposition  with irregularly shaped subdomains such as those obtained by graph partitioning algorithms starting from a fine tessellation. 	}


\subsection*{Notation: Global versus local degrees of freedom}
 In the following we will need to single out different subsets of the set $\Upsilon$ of degrees of freedom. To this aim we start by letting
\[
\mathcal{Y} = \{ i: \ \support(\dofi) \subseteq \Gamma\}
\]
denote the set of indexes $i$ such that  the degree of freedom $\dofi$ is supported on the interface $\Gamma$ of the decomposition. 
For each subdomain $\Omega^\ell$ we let $\Yl \subset \mathcal Y$
\[
\Yl = \{i: \support(\dofi) \subset \partial\Omega^\ell\},
\]
denote the set of indexes of those degrees of freedom supported in $\partial \Omega^\ell$. 

 Analogously, for each macro edge $E$ and for each  macro face $F$  we let
\[ \YE= \{ i: \ \support(\dofi) \subseteq \bar E \}, \qquad \YF = \{ i: \ \support(\dofi) \subseteq \bar F \}, \] 
denote the set of indices of those degrees of freedom supported on $\bar E$ and $\bar F$.
Moreover, we let $\cW$ be the set of indices of those degrees of freedom that are, respectively, supported on the wirebasket $\wirebasket$, 
\[
\cW = \{ i : \ \support(\dofi) \subset \wirebasket \}
\]
and we let
\[\Wl = \cW \cap \Yl, \qquad
\WF = \cW \cap \YF = \{i: \ \support(\dofi) \in \partial F\},
\]
respectively, denote the set of indices of those degrees of freedom that are supported on  the edges of $\Omega^\ell$ and on $\partial F$.

\

Finally, we let $\mathcal{X}$ and $\Xl$ denote, respectively, the set of indices of those degrees of freedom supported in the set $\xpoints$ of cross points (vertices of the subdomain decomposition), respectively, in the set $\xpoints \cap \bar \Omega^\ell$ of vertices of the subdomain $\Omega^\ell$: 
\[\mathcal{X} = \{i: \support(\dofi) \subset \xpoints\}, \qquad \Xl = \mathcal{X} \cap \Yl.\]

Conversely, for each degree of freedom $\dofi$ supported on the skeleton $\Gamma$ of the decomposition we let $\Ni$ denote the set of indices of those subdomains whose boundary the support of $\dofi$ lies on:
\[ \Ni = \{\ell: \support(\dofi) \subset \partial\Omega^\ell\},\qquad n_i = \#(\Ni). \] 
Observe that, for $\ell \in \Ni$ we can, in a natural way, give meaning to the expression $\dofi(w^\ell_h)$ for all $w^\ell_h \in {V_h^\ell}$.
Finally, for each vertex $\vertex \in \xpoints$, edge $E \in \EH$, and face $\face \in \FH$ we can also define the set $\NV$, $\NE$,  and $\NF$  of the indices of those subdomains that share, respectively, $\vertex$, $E$  and $F$ as a vertex, edge or face:
\[ \NV = \{
\ell: \ \vertex  \subset \partial\Omega^\ell\}, \qquad 
\NE = \{ \ell: \ E \subset \partial\Omega^\ell\}, \qquad \NF = \{ \ell: \ F \subset \partial\Omega^\ell\}.
\]
Remark that for all $i \in \YE \setminus \cX$ we have $\Ni = \NE$.

	\section*{Notation: Scaled norms and seminorms}
	In the following we will make use of suitably scaled norms for the Sobolev spaces defined on faces and edges of the subdomains. More precisely, 
	letting $D$ denote any $d$ dimensional domain ($d=1,2,3$) we set
	\begin{gather}\label{scalednorms1}
	\| w \|^2_{L^2(D)} = H^{-d} \int_D | w |^2, \qquad | w |_{H^1(D)} = H^{2-d} \int_{D} | \nabla w |^2,\\
	\label{scalednorms2}
	| w |_{H^s(D)} = H^{2s-d} \int_D \int_D \frac{| w(\sigma) - w(\tau) |^2}{|\sigma - \tau |^{d+2s}}\, d\sigma\, d\tau.
		\end{gather}

	For $D \subset \mathbb{R}^2$ bounded domain, we will also consider the spaces 
	$H^s_0(D)$ ($s \not= 1/2$) and 
	$H^{1/2}_{00}(D)$ of those functions $w$  whose extension by zero $\widehat w$  is in 
	$H^s(\mathbb{R}^2)$ ($s \not= 1/2$) and 
	$H^{1/2}(\mathbb{R}^2)$,  respectively, which we will equip with the norm
	\begin{equation}\label{scalednorms3}
\| w \|_{H^s_0(D)} = | \extension{w} |_{H^s(\mathbb{R}^2)} \qquad \text{and} \qquad \| \extension{w} \|_{H^{1/2}_{00}(D)} = | \extension{w} |_{H^{1/2}(\mathbb{R}^2)}.
	\end{equation}
	
\subsection*{Domain decomposition and FETI-DP preconditioner} The subdomain spaces $V_h^\ell$ and bilinear forms $a_h^\ell:V_h^\ell \times V_h^\ell \to \mathbb{R}$ are defined, as usual, as
\[ 
V_h^\ell = {V_h}|_{\oml}, \qquad a_h^\ell(u_h,v_h) = \sum_{K\in \Thk} a_h^K(u_h,v_h).
\label{ah}
\]
In view of (\ref{elemequiv}) we immediately obtain that for all $u_h, v_h \in V_h^\ell$
\[
a_h^\ell(u_h,v_h) \lesssim {\rho_\ell H }| u_h |_{1,\oml}  | v_h |_{1,\oml}, \qquad a_h^\ell(u_h,u_h) \simeq {\rho_\ell H} | u_h |_{1,\oml}^2.
\]
Solving Problem \ref{discrete_full} is then reduced to finding $u_h = (u_h^\ell)_\ell \in \prod V_h^\ell$ minimizing
\[
J(u_h) = \frac 1 2 \sum_{\ell} a^\ell_h(u_h^\ell,u_h^\ell) -  \sum_K \int_K g^K_h u_h,
\]
subject to a continuity constraint across the interface. 
\

The FETI-DP 
preconditioner is then constructed according to the same strategy used in the finite element case, which we recall, mainly to fix some notation. We let
\[
\mathring V_h = \prod \mathring V_h^\ell \quad \text{ with }\quad \mathring V_h^\ell = V_h^\ell \cap H^1_0(\oml),
\]
and
\begin{equation}\label{Wspace}
W_h = \prod_{\ell}W_h^\ell  \quad \text{ with }\quad W_h^\ell = {V_h^\ell}|_{\partial\oml}.
\end{equation}
Moreover, for each macro face $F$ and macro edge $E$ we let 
$$W_h^F = V_h|_F\qquad \text{and} \qquad W_h^E=V_h|_E$$
 denote, respectively, the trace on $F$ and $E$ of $V_h$. On $W_h$ we define a norm and a seminorm:
\[
\normWh{w_h}^2 = \sum_{\ell} \rho^\ell \| w_h^\ell \|_{H^{1/2}(\partial\oml)}^2, \qquad \snormWh{w_h}^2 = \sum_{\ell}\rho^\ell | w_h^\ell |_{H^{1/2}(\partial\oml)}^2.
\]

As usual, we define local discrete lifting operators $\Lifthl: W_h^\ell \to V_h^\ell$ as
\begin{eqnarray}
a^\ell_h (\Lifthl w_h,v_h)=0 \quad \forall v_h \in \mathring V_h^\ell, \qquad
\Lifthl w_h = w_h \quad \ \text{ on }\partial\oml.
\end{eqnarray}
The following proposition holds.
\begin{proposition}
	$\Lifthl$ is well defined, and it verifies 
	\[
	| \Lifthl w_h |_{H^1(\oml)} \simeq | w_h |_{H^{1/2}(\partial\oml)}.
	\]
\end{proposition}

\begin{proof}
	We start by recalling that 	there exists a linear operator $\Pscottzhang: H^1(\Omega^\ell) \to V^\ell_h$ such that, for $v \in H^{1+s}(\Omega^\ell)$, $0 \leq s \leq 1$, one has (\cite{Cangianietal17})
	\begin{equation}\label{SZ}
	\| v - \Pscottzhang v \|_{L^2(\Omega^\ell)} + \frac{h}{H} | v - \Pscottzhang v |_{H^1(\Omega^\ell)} \lesssim 
	\left(\frac{h}{H}  \right)^{1+s}  | v |_{H^{1+s}(\Omega^\ell)} 
	\end{equation}
(note that we are using scaled norms, see (\ref{scalednorms1})--(\ref{scalednorms3})).	Moreover $\Pscottzhang$ is constructed in such a way that if $v|_{\partial\Omega^\ell} \in W_h^\ell$ one has $\Pscottzhang v = v$ on $\partial\Omega^\ell$.
	 In view of this result it is not difficult to construct an operator $L_h^\ell: W_h^\ell \to V_h^\ell$ satisfying
	$L^\ell_h w_h |_{\partial\oml} = w_h$ and 
	\begin{equation}\label{stab1} | L^\ell_h w_h |_{H^1(\oml)} \leq | w_h |_{H^{1/2}(\partial\Omega^\ell)}.
	\end{equation}
	$L^\ell_h$ can be for instance defined as $\Pi_{SZ}$ applied to the harmonic lifting of $w_h$;~\eqref{stab1} follows then from the stability of the  harmonic lifting and of $\Pscottzhang$, exactly as in the two dimensional case \cite{FETI_VEM_2D}. We can then write
	\begin{multline*}
	| \Lifthl w_h - L^\ell_h w_h|_{H^1(\oml)}^2 \lesssim {H^{-1} \rho^{-1}_\ell} a_h^\ell( \Lifthl w_h - L^\ell_h w_h, \Lifthl w_h - L^\ell_h w_h) \\  = - {H^{-1} \rho^{-1}_\ell} a_h^\ell(  L^\ell_h w_h, \Lifthl w_h - L^\ell_h w_h) \lesssim | L^\ell_h w_h |_{H^1(\oml)} 	| \Lifthl w_h - L^\ell_h w_h|_{H^1(\oml)}.
	\end{multline*}
After dividing both sides by $\rho_\ell | \Lifthl w_h - L^\ell_h w_h|_{H^1(\oml)}$,	in view of (\ref{stab1}) it is easy to conclude by applying a triangular inequality. 
\end{proof}



%
%
%

\

Let now $\widehat W_h \subset W_h$ denote the subset  of 
 functions which are single valued across $\Gamma$:
\begin{align}\label{What}
\widehat W_h &=  \{ w_h \in W_h:\ \forall \ell,k,\  \forall x \in \bar \Omega^\ell \cap \bar \Omega^k, \ w_h^\ell(x) = w_h^k(x)\}  \\[2mm]
& =   \{ w_h \in W_h: \forall \ell,k,\ \forall i \in \Yl \cap \mathcal{Y}^k,\ 
\dofi(w_h^\ell) = \dofi(w_h^k)
\}. \nonumber
\end{align}

\

For $w_h = (w_h^\ell)_\ell \in W_h$ we let $\Lifth(w_h) = (\Lifthl w_h^\ell)_\ell$, 
so that 
 $V_h$ is split as 
\begin{equation}\label{splitting}
V_h = \mathring V_h \oplus \Lifth \widehat W_h.
\end{equation}

\

We next define the bilinear form $s:W_h \times W_h \to \mathbb{R}$ as
\[
s(w_h,v_h) := \sum_\ell a^\ell_h (\Lifthl w_h,\Lifthl v_h).
\]
The proof of the following proposition, where the factor $H$ in bounds \eqref{contcoercschur} stems from using norms scaled as in \eqref{scalednorms2},  is trivial.
\begin{proposition} For all $v_h, w_h \in W_h$ we have
	\begin{equation} \label{contcoercschur}
	s(v_h,w_h) \lesssim H | w_h |_{1/2,*} | v_h |_{1/2,*}, \qquad s(w_h,w_h) \gtrsim H | w_h |_{1/2,*}^2.
	\end{equation}	
\end{proposition}

\

The solution $u_h$ of Problem \ref{discrete_full} is then split as 
$u_h = \mathring u_h + \Lifth w_h$, where $\mathring u_h$ and $w_h$ are the solutions of the following two independent problems.
\begin{problem}\label{interior} Find $\mathring{u}_h \in \mathring V_h$ such that for all $v_h \in \mathring V_h$ 
	\[
	a_h (\mathring u_h,v_h) =  \sum_K \int_K g^K_h v_h.
	\]
\end{problem}

\begin{problem}\label{schur} Find $w_h \in \hW$ such that for all $v_h \in \hW$
	\[
	s(w_h,v_h) =  \sum_K \int_K g^K_h \mathcal{L}_h v_h.
	\]
\end{problem}

\

Exactly as in the finite element case, the design of different versions of the FETI-DP method will rely on the choice of a subspace  $\widetilde W_h$ of $W_h$ whose elements have some degree of continuity on $\Gamma$, ensuring that the restriction to $\widetilde W_h$ of the bilinear form $s$ is coercive (which is equivalent to asking that the seminorm $| \cdot |_{1/2,*}$ is a norm on $\widetilde W_h$). We recall that, while in two dimensions the space $\widetilde W_h$ can be defined as the subspace of functions continuous at the vertices of the subdomains, it is known that this is not sufficient to get a quasi-optimal result in the three dimensional case, so that we also need to impose continuity of either edge or face averages (or both). While later in the paper we will analyze the different possible choices, for now we will only assume that $s$ is coercive on $\widetilde W_h$ (or, equivalently, that $a_h$ is coercive on $\Lifth \widetilde W_h$), and that $\widehat W_h \subseteq \widetilde W_h$.

\

Following the approach of \cite{brennersung}, we introduce the operators $\hSop:\hW \to \hW'$ and $\tSop: \tW \to \tW'$ defined, respectively, as
\begin{equation}\label{S_tilde_hat}
\langle \hSop w_h, v_h \rangle = s(w_h,v_h)\ \forall v_h \in \hW, \qquad  \langle \tSop w_h, v_h \rangle = s(w_h,v_h)\ \forall v_h \in \tW,
\end{equation}
and we let $\Rop: \hW \to \tW$ denote the natural injection operator.  We observe that  
\[
\hSop = \Rop^T \tSop \Rop.
\]
Problem \ref{schur} becomes 
\begin{equation}\label{system_BDDC}
\hSop u =\widehat g \qquad {\text{ with }\quad \langle \widehat g, w_h \rangle =  \sum_{K \in \Th} \int_K g^K_h \Lifth w_h }.
\end{equation}
%
%

As in \cite{Mandel_algebraic,Klawonn.Widlund.06}, 
we define a scalar product $d:W_h \times W_h \to \mathbb{R}$ in terms of the degrees of freedom $\dofi$, $i \in \mathcal{Y}$ 
 as
\begin{equation}\label{defscald}
d(w_h,v_h) =  \sum_\ell \sum_{i \in \Yl} d^{\ell,i} 
{\dofi(w_h^\ell)}{\dofi(v_h^\ell)},
\end{equation}
where, for $i \in \Yl$, the scaling coefficient $d^{\ell,i}$ is defined as 
\begin{equation}\label{dli}
d^{\ell,i}
= 
\frac{\rho_\ell^\gamma}{\sum_{j\in \Ni } \rho_j^\gamma}.   
\end{equation}

Next, we  introduce the projection operator $\ED: W_h \to \hW$, orthogonal  with respect to the scalar product $d$:
\begin{equation}\label{ED}
d(\ED w_h,v_h) = d(w_h, v_h)\quad  \ \forall v_h \in \hW.
\end{equation}
%
%
%

Following \cite{brennersung}, we introduce the quotient space $\tW / \hW$. We let $\Lambda_h = (\tW / \hW)'$ denote its dual and we let
$\Bop: \tW \to \tW / \hW= \Lambda_h'$ denote the quotient mapping, defined as 
\[
\Bop w_h = w_h +\hW.
\]

Observe that two elements $w_h$ and $v_h$ are representative of the same equivalence class if and only if they have the same jump across the interface: for all macro facs $F \in \FH$ with $F \subseteq\partial \Omega^m \cap \partial\Omega^\ell$,  $w_h^m - w_h^\ell = v_h^m - v_h^\ell$. We can then identify $\tW / \hW$ with the set of jumps of elements of $\widetilde W_h$. The quotient map $\Bop$ can then be interpreted as the operator that maps an element of $\widetilde W_h$ to its jump on the interface.
Clearly
\[ \hW = \ker(\mathcal{B}) = \{  w_h \in \tW: \ b(w_h,\lambda)=0 \ \forall \lambda\in \Lambda_h \},\]
where $b: \tW \times \Lambda_h \to \mathbb{R}$ is defined  as $b(w_h,\lambda_h) = \langle \Bop w_h,\lambda_h \rangle$.
Problem \ref{discrete_full} is then equivalent to the following saddle point problem: find $w_h \in \tW$, $\lambda_h \in \Lambda_h$ solution to
\begin{equation}\label{eq:discrete_saddle}
\tSop w_h - \Bop^T \lambda_h = \widehat g,\qquad \Bop w_h = 0.
\end{equation}
Using the first equation in \nref{eq:discrete_saddle} to express $w_h$ as a function of $\lambda_h$, we eliminate the former unknown and finally reduce the solution of Problem \ref{schur} to the solution of a problem in the unknown $\lambda_h \in \Lambda_h$ of the form 
\begin{equation}\label{system_FETI}
\Bop \tSop^{-1}\Bop^T \lambda_h = -\Bop\tSop^{-1}\widehat g.
\end{equation}

 Letting now $\Bop^+ : \Lambda'_h \to \tW$ be any right inverse of $\Bop$ (for $\eta \in \Lambda_h' = \tW / \hW$, $\Bop^+ \eta$ is any element in $\tW$ such that $\Bop \Bop^+ \eta = \eta
 $), we set
\begin{equation}\label{defBop}
\Bop^T_D = (\IdtW-\ED) \Bop^+,
\end{equation}
where $\IdtW$ denotes the identity operator in the space $\tW$.
Recall that, as in the two dimensional case,  the definition of $\Bop_D^T$ is independent of the actual choice of the operator $\Bop^+$. 
Indeed $\Bop (\Bop_1^+ \eta - \Bop_2^+ \eta)  = \eta - \eta =0$ implies $\Bop_1^+ \eta - \Bop_2^+ \eta \in \hW$ and hence $(\IdtW-\ED)\Bop_1^+ \eta - (\IdtW-\ED) \Bop_2^+ \eta = 0$.

\

We finally let the FETI-DP preconditioner  $\FETI: \Lambda'_h \to \Lambda_h$  be defined by 
\begin{equation}\label{prec_FETI}
\FETI = \Bop_D \tSop \Bop_D^T.
\end{equation}

	\

Let us now come to the choice of the space $\widetilde W_h$, or, equivalently, to the choice of the so called {\em primal} degrees of freedom (d.o.f.s), which are taken as single valued, thus directly imposing the corresponding degree of continuity across the interface, whereas continuity for the remaining {\em dual} (d.o.f.s) is imposed via Lagrange multipliers in $\Lambda_h$. Exactly as in the finite element case there are several possibilities. Letting
\begin{align}
\label{Wtilde_V}
\rosa{\widetilde W_h^{(V)}} & = \{ w_h \in W_h: 
w_h^\ell (\vertex)=w_h^m(\vertex) , \,\, \forall \vertex \in \xpoints, \ell,m \in \NV\}, \\
\rosa{\widetilde W_h^{(E)}} & = 
\{ w_h \in W_h :   \int_E w_h^\ell = \int_E w_h^k,\, \, \forall E \in  \EH, \ell,k \in \mathcal{N}_E \},\nonumber \\
\rosa{\widetilde W_h^{(F)}} & =  \{ w_h \in W_h : 
\int_F w_h^\ell = \int_F w_h^k,  \,\, \forall F \in  \FH, \ell,k \in \mathcal{N}_F \}, \nonumber 
\end{align}
  denote  the subset of $W_h$ of traces of functions which, respectively, are continuous at cross-points, have the same average at all edges, and have same average at all faces, 
different  choices for $\widetilde W_h$ can be considered,  resulting in different versions of the FETI-DP algorithms. More precisely we have
\medskip
\begin{itemize}[align=right,labelwidth=1.5cm] \setlength\itemsep{.5em}
	\item[{\sc E:}] $\widetilde W_h = \widetilde W_h^{(E)}$  (primal d.o.f.s are the edge averages);
	\item[{\sc F:}] $\widetilde W_h =\rosa{\widetilde W_h^{(F)}}$  (primal d.o.f.s are the  face averages);
	\item[{\sc VE:}] $\widetilde W_h = \rosa{\widetilde W_h^{(V)}} \cap\rosa{\widetilde W_h^{(E)}}$ (primal d.o.f.s are the values at cross points and the edge averages);
	\item[{\sc VF:}] $\widetilde W_h = \rosa{\widetilde W_h^{(V)}} \cap\rosa{\widetilde W_h^{(F)}}$  (primal d.o.f.s are the values at cross points and the face averages);
	\item[{\sc EF:}] $\widetilde W_h =\rosa{\widetilde W_h^{(E)}} \cap\rosa{\widetilde W_h^{(F)}}$  (primal d.o.f.s are edge and face averages);
	\item[{\sc VEF:}] $\widetilde W_h = \rosa{\widetilde W_h^{(V)}} \cap\rosa{\widetilde W_h^{(E)}} \cap\rosa{\widetilde W_h^{(F)}}$ (primal d.o.f.s are the values at cross points and the face and edge averages).
\end{itemize}

\

\subsection{The main theorem}

We can now state the main result of this paper.
\begin{theorem}\label{bound_FETI-DP} Letting $\kappa$ denote the condition number of the matrix corresponding to the operator $\FETI(\Bop \tSop^{-1} \Bop^T)$, depending on the choice of the space $\widetilde W_h$ we have the following bounds:
	\begin{description}[leftmargin=*] 
		\item[{Algorithms   E/EF:}]	If $\widetilde W_h \subseteq\rosa{\widetilde W_h^{(E)}}$  then
	\[		\kappa\lesssim \left(1+\logHhd+\CjumpE \right) \faclogd; \]
		
		\item[{Algorithm F:}] 	If $\widetilde W_h \subseteq\rosa{\widetilde W_h^{(F)}}$
		then 
	\[	\hspace{.1cm}  
		\kappa \lesssim  \left (1+\logHhd+\CjumpF\right)\faclogd;
		\]
		
		\item[{Algorithms  VE/VEF:}] 	If $\widetilde W_h \subseteq\rosa{\widetilde W_h^{(V)}} \cap\rosa{\widetilde W_h^{(E)}} $  then
	\[
		\kappa \lesssim \faclogd^2; 
		\]
		\item[{Algorithm VF:}]  	If $\widetilde W_h \subseteq\rosa{\widetilde W_h^{(V)}} \cap\rosa{\widetilde W_h^{(F)}} $   then
	\[
		\kappa \lesssim \left( 1+\logHhd+\CjumpEF\right )\faclogd,
		\]
		
	\end{description}
	where $\CjumpE$, $\CjumpF$, and $\CjumpEF$ are constants depending on the diffusion coefficient $\rho$  that satisfy
	\begin{equation}\label{rel:tol}
	0 < \CjumpE,\CjumpF,\CjumpEF \leq \max \rho / \min \rho, \qquad
	\CjumpE \leq \CjumpF, \qquad \CjumpEF \leq \CjumpF.\end{equation}
\end{theorem}

As it happens for the finite element case,
all the above choices lead to a quasi-optimal (in terms of dependence on $h$ and $H$) preconditioning, whereas robustness with respect to the jumps in the coefficient $\rho$ is achieved for algorithms {\sc VE} and {\sc VEF}.

\begin{remark}
	The precise definition of the constants $\CjumpE$, $\CjumpF$ and $\CjumpEF$ (which will be detailed in Section \ref{sec:FETI_prec}) is the same as in the finite element case \cite{Toselli.Widlund}, and it is quite technical.  We would like to remark that the bounds \nref{rel:tol} are often quite pessimistic, as we will see in Section \ref{sec:numerical}.
	\end{remark}

\rosa{	
	\begin{remark} We made the assumption that the subdomains are shape regular (point 2 of Assumption \ref{ass:macroel}) in order to have some of the inequalities needed in the proof of Theorem \ref{bound_FETI-DP}. More precisely, such an assumption yields the validity of the Poincar\'e inequality on the macro faces,
		of the trace inequalities on the subdomains and on the macro faces, and of the injection bound \eqref{s0tos}. While Poincar\'e and trace inequalities are valid for a much larger class of domains, including the ones obtained by applying graph partitioning to a fine tessellation, in such a case the constants are a priori dependent on $h$. In order to extend the theoretical result to the irregular case, we would need to provide a partitioning recipe, resulting in a subdomain decomposition for which such bounds can be proven uniformly in $h$. A relevant case in which this is indeed possible, is the case where the ``roughness'' of the subdomains derives from properties of the continuous problems, such as, for instance, when the subdomains are irregular in order to fit a (fixed) irregular boundary or interior interface. More in general, we believe this is still possible if the partitioning is performed in such a way that the rough edges and faces are ``as straight as possible'', however,  it is outside the scope of this paper to deal with such an issue, which we plan to address in the future.	
		\end{remark}}

\begin{remark}  The assumption that the tessellation and the domain decomposition are quasi--uniform can be dropped out, provided that we replace the ratio $h/H$ with the worst possible instance of such a ratio, that is,  $\max_\ell \max_{K \in \Thk} h_K/H_\ell$. Conversely, the assumption that the elements of the tessellation are shape regular (and more precisely, that all the faces of the tessellation lying on the interface are shape regular) plays a key role in the proof, as it is needed for the bounds in Section \ref{sec:VEMprop} to hold.	
	\end{remark}


\newcommand{\matrice}[1]{\mathbb{#1}}
\newcommand{\vettore}[1]{\mathbb{#1}}

\subsection{Implementation of the FETI method in the VEM context}

Contrary to what happens for finite elements, VE discrete functions are explicitly known only on the edges of the subdomains, where they are piecewise degree $\p$ polynomials. On the faces and within the subdomains the VEM basis functions are not explicitly known, and all quantities needed for the implementation of the preconditioner  (but also of the stiffness matrix and load vector) have to be retrieved in terms
of the 
d.o.f.s \labeldofv, \labeldofe, \labeldoff, and \labeldofK, by exploiting the definition of the spaces $\Vfp$ and $\VKp$ (in the terminology of VEM, they must be {\em computable}). Extended details on the computability of the stiffness matrix and the right hand side can be found in \cite{hitchVEM}. Let us concentrate here on the quantities needed for implementing the FETI-DP preconditioner in the VEM context.

As in the finite element case, the FETI-DP method is carried out by, at first, assembling local systems $\matrice{S}^{(\ell )}\vec{w}^{(\ell)} = \widehat{\vec{g}}^{(\ell)}$, where the matrix $\matrice{S}^{(\ell )}$ is the Schur complement obtained from the local stiffness matrix by eliminating the interior degrees of freedom.  The efficient implementation of the FETI-DP method relies on expressing the elements of $w_h$ in a suitable basis, singling  the primal degrees of freedom out, so that imposing $w_h \in \widetilde W_h$ reduces to requiring that the corresponding entries in the unknown vector (the primal unknowns) are single valued.
Equation \eqref{system_FETI} then becomes
\begin{equation}\label{FETI_algebraic}
\matrice{B} \widetilde{\matrice{S}}^{-1} \matrice{B}^T \boldsymbol{\lambda} = \matrice{B} \widetilde{\matrice{S}}^{-1} \widetilde{\vec{g}}, 
\end{equation}
where $\matrice{B}=[\matrice{B}^{(1)},\cdots,\matrice{B}^{(L)}]$ has entries in $\{-1,0,1\}$, and where the matrix $\widetilde{\matrice{S}}$ and the vector $\widetilde{\vec{g}}$ are obtained from the local matrices $\matrice{S}^{(\ell)}$ and load vectors $\widehat{\vec{g}}^{(\ell)}$  by partial assembly in the primal unknowns. The operators involved in equation \eqref{system_FETI} and in the definition \eqref{prec_FETI} of the preconditioner can be expressed in a block form, easily allowing us to reduce the evaluation of their action (in particular the action of $\widetilde{\matrice{S}}^{-1} $) to the solution of local problems, plus a small coarse problem involving the primal degrees of freedom. The linear system \eqref{FETI_algebraic} is solved by a preconditioned conjugate gradient method, with preconditioner given by \eqref{prec_FETI}, whose action involves  the suitably combined (see \eqref{prec_FETI} and \eqref{defBop}) actions of $\widetilde{\matrice{S}}$, any right pseudoinverse $\matrice{B}^+$ of $\matrice{B}$,  and the algebraic realization of the projector $\ED$, whose explicit expression is given by \eqref{defED} in Section \ref{sec:FETI_prec}.

\

Implementing the change of basis from the canonical VEM basis to the new basis entails the need to evaluate the primal degrees of freedom for any given discrete functions, and to explicitly express them as a function of the degrees of freedom \labeldofv, \labeldofe\ and \labeldoff. 
The primal degrees of freedom involved in the definition of $\rosa{\widetilde W_h^{(V)}}$ and $\rosa{\widetilde W_h^{(E)}}$ (vertex values and edge averages) are easily computed, as discrete functions are explicitly known  on the edges of the subdomains. Let us then consider the face averages, involved in the definition of $\rosa{\widetilde W_h^{(F)}}$. 
In the case  $\p=1$, in order to compute the face average of functions $w_h \in W_h$ we need to resort to the definition (\ref{defspacek}) of the space $V^\face_1$, by which we have $\int_\face w_h =  \int_\face \Pi^{\nabla}_\face w_h$.
 We recall that, thanks to (\ref{Greenface}), $\Pi^{\nabla}_\face w_h$ is computable in terms of the (known) value of $w_h$ on $\partial \face$.
For $\p > 1$ the computation is simpler, as the degrees of freedom \labeldoff~ provide direct access to the integral over any face $\face$ of $w_h^\ell$ times any polynomial of degree less than or equal to $\p-2$, and, in particular, times the constants. Observe that, in view of the FETI implementation, it might be advisable to include a  constant function among the elements of the basis $\BasePoly_{\p-2}$.

\newcommand{\WhFo}{\mathring W_h^F}

\section{Proof of Theorem \ref{bound_FETI-DP}}\label{sec:FETI_prec}

We start by remarking that, by construction, we are in the framework of \cite{MandelSousedik}. In particular we have the identity
\[
\Bop_D^T \Bop + \ED = \IdtW.
\]
In fact, it is not difficult to see that for all $w_h \in \tW$ we have that $(\IdtW  - \Bop_D^T \Bop)w_h \in \hW$ and that we have \[d((\IdtW  - \Bop_D^T \Bop)w_h,v_h) = d(w_h , v_h)  \quad \text{ for all }v_h \in \hW.\]
Then we have \cite{MandelSousedik} 
\begin{equation}\label{minimal1}
 \kappa \lesssim \max_{w_h \in \tW} \frac {s(\Bop_D^T\Bop w_h, \Bop_D^T\Bop w_h)} {s(w_h,w_h)} \simeq  \max_{w_h \in \tW} \frac {s(\ED w_h, \ED w_h)} {s(w_h,w_h)}. 
\end{equation}
In order to have a bound on the condition number, we then only need to bound $s(\ED w_h, \ED w_h) \simeq H | \ED w_h |_{1/2,*}^2$ in terms of $s(w_h,w_h)\simeq H |  w_h |_{1/2,*}^2$.

\

To start, let us recall some functional inequalities that will be useful in the following \cite{Bertoluzza.Falletta2011,Bertoluzza.Falletta.new}. Let $F$ be a shape regular polygon (in the following $F$ will be a face of one of the subdomains). Then, for all $\eta \in H^s(F)$, $1/2< s
\leq 1$, we have, uniformly in $s$, the following trace inequalities, 
\begin{equation}\label{trace2}
\| \eta \|_{H^{s-1/2}(\partial F)} \lesssim \frac 1 {\sqrt{2s -
		1}} \| \eta \|_{H^s(F)}, \qquad  | \eta
|_{H^{s-1/2}(\partial F)} \lesssim \frac 1 {\sqrt{2s - 1}} | \eta
|_{H^s(F)}.
\end{equation}

On the other hand, 	we recall that for $s < 1/2$ the two spaces $H^s(F)$ and $H^s_0(F)$ coincide, and 
the two corresponding norms are equivalent. However, the constant in the
equivalence depends  on $s$ and it explodes as 
$s$ converges to $1/2$. More precisely, for all $\eta \in H^s(F)$, $0\leq s < 1/2$, and for all $\alpha \in \mathbb{R}$ it holds that
\begin{equation}\label{s0tos}
\| u \|_{H^{s}_0(F)} \lesssim \frac 1 {1/2 - s} \| u - \alpha \|_{H^{s}(F)} + \frac 1 {\sqrt{1/2-s}} | \alpha |,
\end{equation}
once again uniformly in $s$ (recall that we are using scaled norms, as defined in Section \ref{sec:VEM}, so that the bounds are uniform in $H$).

We also observe that, thanks to the inverse inequality (\ref{inverse2}) and to the scaling of the norms, by using a standard space interpolation technique it is not difficult to prove that for all $r,s \in [0,1]$ with $r<s$, and for all $w_h \in W_h^F$, (resp. $w_h \in W_h^\ell$) it holds that 
\begin{equation}\label{inverses} 
\| w_h \|_{H^s(F)} \lesssim \left(\frac h H\right)^{r - s} 	\| w_h \|_{H^r(F)}, \qquad \| w_h \|_{H^s(\partial \Omega^\ell)} \lesssim \left(\frac h H\right)^{r - s} 	\| w_h \|_{H^r(\partial \Omega^\ell)}.
\end{equation}
An analogous bound holds for the norms in the spaces $H^s_0(F)$ and $H^r_0(F)$ (with the usual care when either $s$ or $r$ is equal to $1/2$), provided $w_h \in W_h^F\cap H^1_0(F)$. In particular, in such cases we have
\begin{equation}\label{inverse00} 
\| w_h \|_{H^{1/2}_{00}(F)} \lesssim \left(\frac h H\right)^{r - s} 	\| w_h \|_{H^r_0(F)}.
\end{equation}

The following proposition holds.
\begin{proposition}\label{prop:trace}
Let $\Omega^\ell$ be a shape regular subdomain and let $F$ be a face of $\Omega^ \ell$. Then for $w_h \in{ W_h}_{|F}$ we have
\[
\| w_h \|_{L^2(\partial F)}  \lesssim  \sqrt{1 + \log(H/h)} \| w_h \|_{H^{1/2}(F)}.  
\]  	
\end{proposition}

\

\begin{proof}
	Using inequalities \eqref{inverses} and (\ref{trace2}) we can write, for  $0< \varepsilon \leq 1/2$ arbitrary,
	\begin{align*}
	\| w_h \|_{L^2(\partial F)} &\leq \| w_h \|_{H^\varepsilon(\partial F)}  \lesssim \frac 1 {\sqrt{\varepsilon}}  \| w_h \|_{H^{1/2+\varepsilon}(F)}  \\ &\lesssim 
	\left(\frac h H \right)^{-\varepsilon} \frac 1 {\sqrt{\varepsilon}}   \| w_h \|_{H^{1/2}(F)} \lesssim \sqrt{1 + \log(H/h)}
	\| w_h \|_{H^{1/2}(F)},
	\end{align*}
	where the last bound is obtained by choosing $\varepsilon = (1+H/h))^{-1}$.	 
	\end{proof}

\

We now prove the following lemma, which is the equivalent, for the Virtual Element Method, of Lemma  5.6 of \cite{Smith91} and Lemma 4.3 of \cite{BPSIV}. 

\begin{lemma}\label{lem:technical}Let $w_h \in W^\ell_h$ and let $\mathring w_h \in W^\ell_h$ be defined by
	$\dofi(\mathring w_h) = 0$ for all $i \in \Wl$, $\dofi(\mathring w_h) = \dofi(w_h)$  for all $i \in \Yl \setminus \Wl$. Then, for all faces $F$ of $\Omega_\ell$  it holds that $\mathring w_h|_F \in H^{1/2}_{00}(F)$ and
	\[
	\| \mathring w_h \|^2_{H^{1/2}_{00}(F)} \lesssim (1 + \log(H/h))^2 \| w_h \|^2_{H^{1/2}(F)}.
	\]
Moreover, if $w_h$ is constant on $F$ then 
		\[
		\| \mathring w_h \|^2_{H^{1/2}_{00}(F)} \lesssim (1 + \log(H/h)) \| w_h \|^2_{H^{1/2}(F)}
		\]  
\end{lemma}

\begin{proof} We let $\WhFo$ be defined as
	\[
 \WhFo = \{ w_h \in W_h^F : \ \dofi(w_h)=0 \text{ for all }i\in \WF \}.
	\]
It is easy to see that $ \dofi(w_h)=0$ for all $i \in \WF$ (degrees of freedom supported on $\partial F$) implies $w_h = 0$ on $\partial F$ and, hence,  $\WhFo \subset H^1_0(F) \subset H^{1/2}_{00}(F)$.	Let $\pi_h: L^2(F) \to W_h^F$ and $\pi_h^0: L^2(F) \to \WhFo$ denote the $L^2$-projection onto $W_h^F$ and onto $\WhFo$, respectively. Recall that for all $u \in H^1(F)$, using \nref{SZ} we have 
	\begin{equation}\label{l2pi}
	\| u - \pi_h u \|_{L^2(F)} \leq \| u - \Pi_{SZ} u \|_{L^2(F)} \lesssim \frac h H | u |_{H^1(F)} 
	\end{equation}
	and, since 
	 $u \in H^1_0(F)$ implies that $\Pi_{SZ}u \in \WhFo$, we also have for all $u \in H^1_0(F)$ 
	\begin{equation}\label{l2pi0}
	\| u - \pi^0_h u \|_{L^2(F)} \leq \| u - \Pi_{SZ} u \|_{L^2(F)} \lesssim \frac h H | u |_{H^1(F)}.
	\end{equation}

	Let now $i_h^0: W_h^F \to \WhFo$ be defined by $\dofi(i_h^0 w_h)  = \dofi(w_h)$  \  for all  $i \in \YF\setminus\WF$, so that, on $F$, $\mathring w_h^\ell =  i_h^0 w_h$.	Remark that, thanks to Lemma \ref{lem:Riesz}, we have
\rosa{	\begin{equation}\label{l2bih0}
	\| i_h^0 w_h \|_{L^2(F)}^2 \lesssim 
\left(
\frac{ h }H 
\right)^2
	\sum_{i\in \YF\setminus\WF} |\dofi(w_h)|^2 \lesssim 
\left(
\frac{ h }H 
\right)^2
	\sum_{i\in \YF} |\dofi(w_h)|^2 \lesssim  \| w_h \|^2_{L^2(F)}.
	\end{equation}

}
	Consider now the operator $\pi_h^1 = i_h^0 \circ \pi_h : L^2(F) \to \WhFo$ obtained by first projecting onto $W_h^F$ and then setting the values at nodes on $\partial F$ to zero. We will  prove  that the restriction of  $\pi_h^1$ to $H^s_0(F)$  is uniformly bounded for all $s < 1/2$, that is, that for all $w \in H^s_0(F)$ we have
	\begin{equation}\label{boundP1}
	\| \pi_h^1 w \|_{H^s_0(F)} \lesssim \|  w \|_{H^s_0(F)} 
	\end{equation}
	with a constant independent of $s$. Then, for $\varepsilon \in
	]0,1/2[$  arbitrary, using \eqref{inverse00} and \eqref{s0tos}, we can write 
	\begin{multline*}
	\| \pi_h^1 w_h \|_{H^{1/2}_{00}(F)} \lesssim \left(\frac{h}{H}\right)^{-\varepsilon} \| \pi_h^1 w_h
	\|_{H^{1/2-\varepsilon}_0(F)} \\  \lesssim
	\left(\frac{h}{H}\right)^{-\varepsilon}
	\| w_h \|_{H^{1/2-\varepsilon}_0(F)} \lesssim 
	   \frac {1}{\varepsilon} \left(\frac{h}{H}\right)^{-\varepsilon} \|
	w_h \|_{H^{1/2-\varepsilon}(F)},
	\end{multline*}
	which, by choosing
	$\varepsilon = 1/ |\log(H/h)|$,
	yields
		\begin{align*}
		\| \pi_h^1 w_h \|_{H^{1/2}_{00}(F)} \lesssim  (1 + \log(H/h)) \| w_h
		\|_{H^{1/2}(F)}.
		\end{align*}
	Observing that for $w_h \in  W_h^F$ we have
	\[
	\mathring w_h^\ell|_F =  i_h^0 w_h|_F = \pi_h^1 w_h|_F,
	\]
	we immediately get the thesis (the result for $w_h^\ell$ constant on $F$   is obtained by setting $\alpha = w_h^\ell$ in \eqref{s0tos}).
	
	\
	
Let us then prove (\ref{boundP1}). 	We easily see that $\pi_h^1$ is $L^2$ bounded: for all $w \in L^2(F)$,
	\[
	\| \pi_h^1 w \|_{L^2(F)} \lesssim \| \pi_h w \|_{L^2(F)} \lesssim \| w \|_{L^2(F)}.
	\]
	On the other hand, observing that $i_h^0 \circ \pi_h^0 = \pi_h^0$, using \eqref{l2bih0} we see that, for $w \in H^1_0(F)$, 
	\begin{multline*}
	\| w - \pi_h^1 w \|_{L^2(F)} = \| w - \pi_h^0 w +
	i_h^0 \pi_h^0 w - i_h^0 \pi_h w \|_{L^2(F)}
	\lesssim
	\| w - \pi_h^0 w \|_{L^2(F)} + \| \pi_h^0 w - \pi_h w \|_{L^2(F)}.
	\end{multline*}
	By adding and subtracting $w$ in the second term on the right hand side and using \nref{l2pi} and \nref{l2pi0}, we obtain, for $w \in H^1_0(F)$,
	\begin{gather*}
	\| w - \pi_h^1 w \|_{L^2(F)} \lesssim
	\| w - \pi_h^0 w \|_{L^2(F)} + \| w - \pi_h w \|_{L^2(F)} \lesssim
	\frac{h}{H} | w |_{H^1(F)}.
	\end{gather*}
	This allows us to prove, by a standard argument, that $\pi_h^1$ is $H^1_0$-bounded. In fact, letting $\Pi_h^1: H^1_0(F) \to \WhFo$ denote the $H^1_0$ projection, for $w \in H^1_0(w)$, using \eqref{inverse2}, adding and subtracting $w$ and then using an approximation bound, we have
	\begin{align*}
	| \pi_h^1 w |_{H^1(F)}& = | \Pi_h^1 w |_{H^1(F)} + | \pi_h^1 w - \Pi_h^1 w |_{H^1(F)} \lesssim
	| w |_{H^1(F)} +\left( \frac{h}{H}\right)^{-1}\| \pi_h^1 w -
	\Pi_h^1 w \|_{L^2(F)} \\
	&\lesssim
	| w |_{H^1(F)} + \left( \frac{h}{H} \right)^{-1} \| \pi_h^1 w -
	w \|_{L^2(F)} + \left( \frac{h}{H} \right)^{-1} \| w - \Pi_h^1
	w \|_{L^2(F)} \\
	& \lesssim | w |_{H^1(F)} + \left( \frac{h}{H} \right)^{-1}
	\left( \frac{h}{H} \right) | w |_{H^1(F)} + \left( \frac{h}{H} \right)^{-1}
	\left( \frac{h}{H}\right ) | w |_{H^1(F)}\lesssim | w |_{H^1(F)}.
	\end{align*}
	%
	The bound (\ref{boundP1}) follows by a standard  space interpolation argument.

\end{proof}

Let us now consider the projector $\ED: W_h \to \widehat W_h$. We have the following lemma.

\begin{lemma}
	For all $w_h \in W_h$ it holds that 
	\begin{equation}\label{mainboundED}
	| \ED w_h |_{1/2,*}^2 \lesssim  (1+\log(H/h))^2 | w_h |_{1/2,*}^2+  
\left( \frac h H \right) 	\DeltaV + \DeltaE + (1 + \log(H/h))\, \DeltaF 
	\end{equation}
	with 
	\begin{eqnarray}
	\DeltaV &=& 
	\sum_{\vertex \in \xpoints} \sum_{\ell, k \in \NV} \min\{\rho_\ell,\rho_k\}  \,| w_h^\ell(\vertex)  - w_h^k(\vertex) |^2,\\
	\DeltaE &=& \sum_{E \in \EH} \sum_{\ell,k \in \NE} \min\{\rho_\ell,\rho_k \} \,|\alpha_\ell^E - \alpha_k^E |^2,\\
	\DeltaF &=& \sum_{F \in \FH} \sum_{\ell,k \in \mathcal{N}_F} \min\{\rho_\ell,\rho_k\} 
	\,| \alpha^F_\ell - \alpha^F_k |^2,
	\end{eqnarray}
	where
	\begin{equation}
	\alpha_\ell^F = | F |^{-1} \int_F w_h^\ell, \qquad \alpha_\ell^E =  | E |^{-1} \int_E w_h^\ell.
	\label{alfaEF}
	\end{equation}
	\end{lemma}

\begin{proof} 
	 It is not difficult to check that, for $i \in \mathcal{Y}$,  we have
	 \begin{equation} \label{defED}
	 \dofi(\ED w_h) = \theta_i^{-1} \sum_{k \in \Ni} \rho_k^\gamma \dofi(w_h^k) \quad \text{ where } \theta_i = \sum_{k \in \Ni} \rho_k^\gamma,
	 \end{equation}
	and that these relations completely define $\ED$, as it is uniquely determined by the value of the degrees of freedom $\dofi$, $i \in \mathcal{Y}$, that is, by the degrees of freedom supported on the interface of the domain decomposition.

		Let now both $w_h$  and $v_h = \ED w_h$ be split as the sum of the contributions of the degrees of freedom supported on the wirebasket, which we will denote by $w_h^\sharp$  and $v_h^\sharp$, respectively, and the contribution of remaining degrees of freedom, which we will denote by $\mathring w_h$ and  $\mathring v_h$, respectively. More precisely, we let $w_h^{\sharp} \in W_h$ and $v^\sharp_h \in \hW$ be defined by
	\[
w_h^\sharp = (w^{\ell,\sharp}_h)_\ell \text{ with }	\dofi(w_h^{\ell,\sharp}) = \begin{cases}\dofi(w^\ell_h), & i \in \Wl,\\  0, & i \in \Yl \setminus \Wl,\end{cases} \quad \text{and} \quad \dofi(v^\sharp_h) = \begin{cases} \dofi(v_h), & i \in \mathcal{W},\\
	0, & i \in \mathcal{Y}\setminus\mathcal{W},\end{cases}
	\]
and we set $\mathring w_h = w_h - w_h^\sharp$ and $\mathring v_h = v_h - v_h^\sharp$.	Remark that
	\[
	v_h^\sharp = \ED w_h^\sharp, \qquad \mathring v_h = \ED \mathring w_h.
	\]

		To start, let us consider the contribution of the faces. 
		We have
		\begin{gather*}
		\rho_\ell | \mathring v_h |^2_{H^{1/2}(\partial\Omega_\ell)} \lesssim \rho_\ell \sum_{F\in \Fl} \|  \mathring v_h \|_{H^{1/2}_{00}(F)}^2.
		\end{gather*}
		
		We recall that for $a,b > 0$ and $\gamma \geq 1/2$ we have $ab^{2\gamma}/(a^\gamma + b^\gamma)^2 \lesssim \min\{ a , b\}$. Let $F$ be a common face of the subdomains $\Omega^\ell$ and $\Omega^k$.  On $F$ we have $\ED \mathring w_h = \theta_F^{-1} (\rho_\ell^\gamma \mathring w_h^{\ell} + \rho_k^\gamma \mathring w_h^{k})$,  where $\theta_F = \rho_\ell^\gamma+ \rho_k^ \gamma$. 
		Let $\Theta_F \in W_h^F$ denote the function whose degrees of freedom assume the following values:
		\[
		\dofi(\Theta_F) = \dofi(1)\quad \forall i \in \YF \setminus \WF, \qquad \qquad \dofi(\Theta_F) = 0\quad \forall i \in \WF,
		\]
	where $\dofi(1)$ stands for the action of the functional $\dofi$ on the constant unit function.	We can write
		\begin{multline*}
		\rho_\ell \| \mathring w_h^{\ell} - \mathring v_h \|_{H^{1/2}_{00}(F)}^2 
		= \rho_ \ell (\theta_F^{-1}\rho_k^ \gamma)^2 \| \mathring w_h^{\ell} -\mathring w_h^{k} \|_{H^{1/2}_{00}(F)}^2 \\[2mm]
		\lesssim \min\{\rho_\ell,\rho_k\} \| \mathring w_h^{\ell} - \alpha^F_\ell \Theta_F 
		+ \alpha^F_k \Theta_F - \mathring w_h^{k} \|_{H^{1/2}_{00}(F)}^2
		+ \min\{\rho_\ell,\rho_k\}  \| (\alpha^F_\ell - \alpha^F_k) \Theta_F \|_{H^{1/2}_{00}(F)}^2
		\\[2mm]
		\lesssim
		\rho_\ell  \| \mathring w_h^{\ell} - \alpha^F_\ell \Theta_F  \|_{H^{1/2}_{00}(F)}^2 + \rho_k  \| \mathring w_h^{k} - \alpha^F_k \Theta_F  \|_{H^{1/2}_{00}(F)}^2 + \min\{\rho_\ell,\rho_k\}  \| (\alpha^F_\ell - \alpha^F_k) \Theta_F \|_{H^{1/2}_{00}(F)}^2.
		\end{multline*}
		We now apply Lemma \ref{lem:technical}, which, thanks to the Poincar\'e inequality, gives us
			\begin{multline*}
				\rho_\ell \| \mathring w_h^{\ell} - \mathring v_h \|_{H^{1/2}_{00}(F)}^2 
			\lesssim 	(1+ \log(H/h))^2	\left(\rho_\ell  | w_h^{\ell}  |_{H^{1/2}(F)}^2 + \rho_k  |  w_h^{k}  |_{H^{1/2}(F)}^2 \right) \\ + \min\{\rho_\ell,\rho_k\} (1+\log(H/h)) |\alpha^F_\ell - \alpha^F_k|^2.
			\end{multline*}
		
		Adding up over all subdomains and over all faces (each face is counted twice) we obtain
		\[
		| \mathring w_h - \mathring v_h |_{1/2,*}^2 \lesssim (1 + \log(H/h))^2 | w_h |_{1/2,*}^2 + (1 + \log(H/h)) \DeltaF.
		\] 
		
\	

We now consider the contribution of the wirebasket. We recall that for $i \in \cW$ ($\dofi$ is supported on the wirebasket $\wirebasket$), the degree of freedom $\dofi$ corresponds to a node $y_i$ and takes the form $\dofi( w ) = w(y_i)$.	Then, for all $i \in \cW$, \eqref{defED} can be rewritten as 
\[
v_h(y_i) = \theta_i^{-1}\sum_{k \in \Ni} \rho_k^\gamma w_h^k(y_i).
 \]
	Using  \eqref{inverses} and \eqref{riesz1d}
	we can write	\begin{equation}\label{boundi}
	\rho_\ell | w_h^{\ell,\sharp} - v_h^\sharp |^2_{H^{1/2}(\partial\Omega^\ell)} \lesssim \rho_\ell \left( \frac h H \right)^{-1} \|  w_h^{\ell,\sharp} - v_h^\sharp \|^2_{L^2(\partial\Omega^\ell)} \lesssim \rho_\ell \left( \frac h H \right)
	\sum_{i \in \Wl} | w_h^\ell(y_i) - v_h(y_i) |^2
	 \end{equation}
	 and 
		\begin{multline}\label{splitxp}
		\rho_\ell 
		| w_h^\ell(y_i) - v_h(y_i) |^2 = \rho_\ell |\theta_i^{-1} \sum_{k\in \Ni} \rho^\gamma_k (w_h^\ell(y_i)  - w_h^k(y_i) )|^2 	
		\\ \lesssim 
		\sum_{k\in \Ni} \rho_\ell(\theta_i^{-1} \rho^\gamma_k)^2 | w_h^\ell(y_i)  - w_h^k(y_i) |^2 		
		\lesssim 	\sum_{k\in \Ni} \min\{\rho_\ell,\rho_k\} | w_h^\ell(y_i)  - w_h^k(y_i) |^2.
		\end{multline}		
Plugging (\ref{splitxp}) into (\ref{boundi})	and adding up over all $\ell$ we obtain
	\begin{multline}
| w_h^\sharp - v_h^\sharp |_{1/2,*}^2 \lesssim \left( \frac h H \right) \sum_\ell \sum_k \sum_{i \in \Wk \cap \Wl} \min\{\rho_\ell,\rho_k\}  | w_h^\ell(y_i)  - w_h^k(y_i) |^2 	\lesssim \\
\left( \frac h H \right) \DeltaV +  \sum_{E \in \EH} \sum_{\ell,k \in \NE} \left( \frac h H \right) \sum_{i \in \YE}\min\{\rho_\ell,\rho_k\}  | w_h^\ell(y_i)  - w_h^k(y_i) |^2.
	\end{multline} 		
	Now, given $E \in \EH$, for $\ell,k \in \NE$ and $i \in \YE$ we can write
	\[
		\min\{\rho_\ell,\rho_k\} 
		| w_h^\ell(y_i) - w_h^k(y_i)|^2 \lesssim \rho_\ell  | w_h^\ell(y_i)  - \alpha^E_\ell |^2 + \rho_k  | w_h^k(y_i) - \alpha^E_k|^2 + \min\{\rho_\ell,\rho_k\}
		   | \alpha_\ell^E - \alpha_k^E |^2,
	\]	
yielding
\begin{multline}
\left( \frac h H \right) \sum_{i \in \YE} \min\{\rho_\ell,\rho_k\}  | w_h^\ell(y_i)  - w_h^k(y_i) |^2 \lesssim
\left( \frac h H \right)\sum_{i \in \YE} \rho_\ell  | w_h^\ell(y_i)  - \alpha^E_\ell |^2\\ + \left( \frac h H \right) \sum_{i \in \YE} \rho_k  | w_h^k(y_i) - \alpha^E_k|^2  + \left( \frac h H \right) \min\{\rho_\ell,\rho_k\} \#(\YE) | \alpha_\ell^E -  \alpha_k^E |^2 \\ \lesssim
\rho_\ell\| w_h^\ell - \alpha_\ell^E \|_{L^2(E)}^2 + \rho_k \| w_h^k - \alpha_k^E \|_{L^2(E)}^2 + \min\{\rho_\ell,\rho_k\} | \alpha_\ell^E -  \alpha_k^E |^2,
\end{multline}
where we used once again \eqref{riesz1d} and the fact that, under the assumptions made on the tessellation, we have that $\#(\YE) \lesssim H/h$. Then
\[
| w_h^\sharp - v_h^\sharp |_{1/2,*}^2 \lesssim \left( \frac h H \right) \DeltaV + \DeltaE + \sum_E \sum_{\ell \in \NE} \rho_\ell \| w_h^\ell - \alpha^E_\ell \|^2_{L^2(E)}. 
\]
We conclude by observing that
\[
 \sum_E \sum_{\ell \in \NE} \rho_\ell \| w_h^\ell - \alpha^E_\ell \|^2_{L^2(E)} \lesssim \sum_{\ell} \sum_{F\in\Fl} \rho_\ell \| w_h^\ell - \alpha^F_\ell \|^2_{L^2(\partial F) }
\]
(we used that $\alpha_\ell^E$ minimizes $\| w_h^\ell - \alpha \|_{L^2(E)}$).
	Applying Proposition \ref{prop:trace} we then obtain 
\[	 \sum_E \sum_{\ell \in \NE} \rho_\ell \| w_h^\ell - \alpha^E_\ell \|^2_{L^2(E)} \lesssim (1+\log(H/h)) \| w_h - \alpha^F \|_{1/2,*}^2 \lesssim (1+\log(H/h)) | w_h  |_{1/2,*}^2.
\]	
\end{proof}

Observe that $\Delta^{\mathcal{X}}$, $\Delta^{\mathcal{E}}$, and $\Delta^{\mathcal{F}}$ vanish provided that $w_h$ belongs to $\widetilde{W}^V$, $\widetilde{W}^E$, and $\widetilde{W}^F$, respectively, so that, depending on the choice of $\widetilde W_h$, some of the terms on the right hand side of  (\ref{mainboundED}) disappear. In order to get a bound for $\ED$ for the different choices of $\widetilde W_h$, we then need to bound the remaining terms in the different cases. We start by comparing, for a given function $w_h$, the average over a face with the average over one of its edges.

\begin{lemma}\label{edge-face}
	For $E$ edge of $F\subset \partial \Omega^\ell$, it holds for $\alpha^E_\ell,\, \alpha^F_\ell$ defined in (\ref{alfaEF}): 
	\begin{equation}
|	\alpha^E_\ell-\alpha^F_\ell|^2 \lesssim  (1+\log(h/H)) |w_h^\ell |^2_{H^{1/2}(F)}.
	\end{equation}
\end{lemma}

\newcommand{\Piconst}{\pi_E}
\newcommand{\PwC}{\mathcal{K}_H}
\begin{proof}
Let $\Piconst$  denote the $L^2(E)$ projection onto the space of constant functions, which is defined by 
\[
\Piconst w = |E|^{-1} \int_E w.
\]
Trivially, such an operator preserves the constants. Then, by using Proposition \ref{prop:trace} and a Poincar\'e type inequality, allowing to bound the $H^{1/2}(F)$ norm with the $H^{1/2}(F)$ seminorm for the average free function $w_h^\ell - \alpha_\ell^F$,  we can write
\begin{multline*}
| \alpha_\ell^E - \alpha_\ell^F |^2 \simeq \| \Piconst (w_h^\ell - \alpha_\ell^F) \|^2_{L^2(E)}
\leq \| w_h^\ell - \alpha_\ell^F \|^2_{L^2(E)}
\lesssim 
(1+\log(H/h)) | w_h^\ell  |^2_{H^{1/2}(F)}.
\end{multline*}
	\end{proof}

We can now bound $\Delta^{\mathcal{X}}$ in terms of either $\Delta^{\mathcal{E}}$  or $\Delta^{\mathcal{F}}$.
\begin{lemma}
	The following inequalities hold:
	\begin{eqnarray}
\label{EboundsV}
	\hH\DeltaV &\lesssim& \CjumpE  \left( (1+\log(H/h) | w_h |_{1/2,*}^2 +   \DeltaE\right),\\
	\label{FboundsV}
		\hH\DeltaV &\lesssim& \CjumpF  (1+\log(H/h) \left(| w_h |_{1/2,*}^2 + \DeltaF\right) 
	\end{eqnarray}
	with  $\CjumpE$, $\CjumpF$ 
	constants depending on the diffusion coefficient $\rho$,  satisfying 
$0 < \CjumpE\leq \CjumpF \leq \frac{ \max \rho} {\min \rho}$. 

\end{lemma}

\newcommand{\Kpath}{N}
\newcommand{\Kstar}{K^*}
\newcommand{\cpath}{\tau}

\begin{proof} We start by proving (\ref{EboundsV}). Let $V\in \Xi$ denote one of the cross points, let $\dofi$ (with $ \in \mathcal{X}$) denote the corresponding d.o.f., and let $\ell,k \in \Ni$. Assume at first that $\Omega^\ell$ and $\Omega^k$ share an edge $E$ having $V$ as one of the vertices. Adding and subtracting $(\alpha_\ell^E - \alpha_k^E)$, using Proposition \ref{prop:trace} as well as a Poincar\'e--inequality for function with vanishing average in a portion of the boundary (allowing us to bound the $H^{1/2}$ norm with the $H^{1/2}$ seminorm), we can write
	\begin{gather*}
	\left(  \frac{h}{H} \right) \min\{\rho_\ell,\rho_k \}| w_h^\ell (V) - w_h^k(V) |^ 2 \lesssim \min\{\rho_\ell,\rho_k \}\| w_h^\ell - w_h^k \|_{L^2(E)}^2  \lesssim \\ \lesssim
	(1+\log(h/H)) (\rho_\ell | w_h^\ell |^2_{H^{1/2}(\partial\Omega^\ell )} + \rho_k | w_h^k |^ 2_{H^{1/2}(\partial\Omega^k )}    ) + \min\{\rho_\ell,\rho_k\} | \alpha_\ell^E - \alpha_k^E |^2.
	\end{gather*}
%
%

\newcommand{\Paths}{\mathcal{P}} 
\newcommand{\PathsE}{\mathbf{P}^{\ell,k}_E}
\newcommand{\PathsF}{\mathbf{P}^{\ell,k}_F}
\newcommand{\lpath}{N}
\newcommand{\maxpath}{K^*}

	Let now $\ell,k \in \Ni$ be two subdomains sharing the vertex $V$ but not an edge. In this case we bound $|w_h^\ell(V) - w_h^k(V)|$ by adding and subtracting a suitable sequence of values $w^n_h(V)$ in such a way that  we fall back into the previous case. 
To this aim we start by introducing the following definitions: 
\begin{itemize}
\item a {\it path  $\Paths$} of length $\Kpath$ is 
any sequence of subdomains $\Omega^{n_0}, \ldots, \Omega^{n_\Kpath}$ such that for all $i$, $\Omega^{n_i}$ and $\Omega^{n_{i+1}}$ share at least a vertex;
\item for a given path $\Paths = (\Omega^{n_0}, \cdots,\Omega^{n_N})$ we set $\tau_{\Paths} =(\min \rho_{n_i})^{-1}$, 
$i\in[0,\dots,N]$;
\item a  path $\Paths= (\Omega^{n_0}, \cdots,\Omega^{n_N})$  {\it connects $\Omega^\ell$ and $\Omega^k$ via edges (resp. via faces)} if $n_0 =\ell$, $n_\Kpath = k$ and for all $i = 1,\ldots,\lpath$ the subdomains $\Omega^{n_i}$ and $\Omega^{n_{i-1}}$ share an edge (resp. a face).
\end{itemize}

Letting $\maxpath$ be the maximum number of subdomains sharing a vertex, we denote by $\PathsE$ (resp. $\PathsF$) the set of paths  of length $\leq \maxpath$ connecting $\Omega^\ell$ and $\Omega^k$ via edges (resp. via faces).
For all paths $\Paths = (\Omega^{n_0}, \cdots, \Omega^\Kpath) \in \PathsE$ we can bound
\[
\min\{\rho_\ell,\rho_k \}| w_h^\ell (V) - w_h^k(V) |^ 2 \lesssim \sum_{j=1}^\Kpath 
\frac{\min\{\rho_\ell,\rho_k\}}{
	\min\{
	\rho_{n_j},\rho_{n_{j-1}}
	\}	
} \min\{
\rho_{n_j},\rho_{n_{j-1}}
\}	  | w^{n_j}_h (V) - w^{n_{j-1}}_h(V) |^2
\]	
and, using the bound for subdomains sharing an edge,  we obtain (\ref{EboundsV}) with 
\[
\CjumpE = \max_{(\ell,k): \, \Omega^\ell,\Omega^k \\ \text{ share a vertex}} \left (\min\{\rho_\ell, \rho_k\} \,\, \tau^{\ell,k}_E\right  ), 
\] 
%
%
%
where $$
\tau^{\ell,k}_E =  \min_{\Paths\in \PathsE} 
\, \tau_{\Paths}.
$$

Bound  (\ref{FboundsV}) with
$
\tau^{\ell,k}_F =  \min_{\Paths\in \PathsF} 
\, \tau_{\Paths}
$ and 
\[
\CjumpF = \max_{(\ell,k): \, \Omega^\ell,\Omega^k \text{ share a vertex}} \left (\min\{\rho_\ell, \rho_k\} \,\, \tau^{\ell,k}_F\right )
\]
%
%
%
is obtained by a similar argument. As $\PathsF \subseteq \PathsE$ (if two subdomains share a face they also share a vertex) we easily get that $\tau_E  \leq \tau_F$.
%
%
%
	\end{proof}

Finally, we bound $\Delta^{\mathcal{E}}$ and $\Delta^{\mathcal{F}}$ in terms of each other. 
\begin{lemma}
	The following bounds hold:
	\begin{eqnarray}
	\label{boundDF}
	\DeltaF &\lesssim & (1+\log(H/h))\, |w_h|_{1/2,*}^2 + \DeltaE,\\
	\label{boundDE}
	\DeltaE  &\lesssim &\CjumpEF \left( (1+\log(H/h))\, |w_h|_{1/2,*}^2 + \DeltaF \right), 
	\end{eqnarray}
		with $\CjumpEF$ a constant that satisfies $\CjumpEF \leq \CjumpF$ .
\end{lemma}
\begin{proof}
	Let us consider the first inequality.
	Let $F$ be a face common to the subdomains $\Omega^\ell$ and $\Omega^k$, and let $E$ be any of the edges of $F$. By Lemma \ref{edge-face} we have
	\begin{gather*}
	\min\{\rho_\ell,\rho_k\} |\alpha_\ell^F-\alpha_k^F|^2 \lesssim \rho_\ell |\alpha_\ell^F-\alpha_\ell^E|^2 + 
	\rho_k |\alpha_\ell^F-\alpha_\ell^E|^2 +  	\min\{\rho_\ell,\rho_k\} |\alpha_\ell^E-\alpha_k^E|^2  \\
	\lesssim  (1+\log(H/h))  \left( \rho_\ell |w_h^\ell|^2_{H^{1/2}(F)} +  \rho_k |w_h^k|^2_{H^{1/2}(F)} \right) +  	\min\{\rho_\ell,\rho_k\} |\alpha_\ell^E-\alpha_k^E|^2.
	\end{gather*}
	In view of the definition of $\DeltaF$ and $\DeltaE$, the bound \nref{boundDF} follows up adding the contribution of all faces.

	Let us now consider the second bound. Let $E$ be an edge, and let $\ell,k\in \mathcal{N}_E$. Assume at first that $\Omega^\ell$ and $\Omega^k$ share a face $F$.  
	Then, as we did for the previous bound, it is not difficult to prove that 
	$$ \min\{\rho_\ell, \rho_k\} |\alpha^E_\ell - \alpha^E_k |^2 \lesssim (1+\log(H/h)) \!
	\left( \rho_\ell |w_h^\ell |^2_{H^{1/2}(F)}  + \rho_k |w_h^k |^2_{H^{1/2}(F)} \right) +
	\min\{\rho_\ell, \rho_k\} |\alpha^F_\ell - \alpha^F_k |^2 .
	$$
	Let us now assume that $\Omega^\ell$ and $\Omega^k$ do not share a face. Proceeding as in the proof of the previous lemma it is easy to see that bound \eqref{boundDE} holds with 
\[
\CjumpEF = \max_{(\ell,k): \, \Omega^\ell,\Omega^k \\ \text{ share an edge}} \left (\min\{\rho_\ell, \rho_k\} \,\, \tau^{\ell,k}_F\right  ).
\]	
As two subdomains that share an edge also share a vertex, we have that $\CjumpEF \leq \CjumpF$.
\end{proof}

\

Now we have all the ingredients to prove Theorem \ref{bound_FETI-DP}.
From (\ref{minimal1})  we get that  to  bound  the condition number, we only need to bound $s(\ED w_h, \ED w_h)$ in terms of $s(w_h,w_h)$. 	
By using Lemma \eqref{EboundsV}, \eqref{FboundsV}, \eqref{boundDF}, and \eqref{boundDE} we can easily obtain that
\begin{align}\label{WE}
&\text{if } \widetilde W_h \subseteq\rosa{\widetilde W_h^{(E)}}, \qquad  |\ED w_h |_{1/2,*}^2 \lesssim ((1+\log(H/h))+\CjumpE) (1+\log(H/h)) | w_h |_{1/2,*}^2;  \\[3mm] 
\label{WF}
&\text{if }\widetilde W_h \subseteq\rosa{\widetilde W_h^{(F)}}, \qquad  |\ED w_h |_{1/2,*}^2 \lesssim ((1+\log(H/h))+\CjumpF) (1+\log(H/h)) | w_h |_{1/2,*}^2 ; \\[3mm]
\label{WEV}
&\text{if } \widetilde W_h \subseteq\rosa{\widetilde W_h^{(V)}}\cap\rosa{\widetilde W_h^{(E)}}, 
 \qquad | \ED w_h |_{1/2,*}^2 \lesssim (1+\log(H/h))^2 | w_h |_{1/2,*}^2; \\[3mm]
%
\label{WFV}
&\text{if } \widetilde W_h \subseteq\rosa{\widetilde W_h^{(V)}}\cap\rosa{\widetilde W_h^{(F)}},   \qquad | \ED w_h |_{1/2,*}^2 \lesssim \left ((1+\log(H/h)) +\CjumpEF\right ) (1+\log(H/h)) | w_h |_{1/2,*}^2. 
\end{align}
As, by \eqref{contcoercschur}, we have that $s(\ED w_h, \ED w_h) \simeq H | \ED w_h |_{1/2,*}^2$ and $s(w_h,w_h) \simeq H | w_h |_{1/2,*}^2$, \eqref{WE}--\eqref{WEV} yield the thesis.

\

\begin{remark} \label{equivalence_BDDC}
	Observe that, since we are in the framework of 	\cite{MandelSousedik}, we have the equivalence of the BDDC preconditioner with the FETI-DP preconditioner. Therefore the analysis presented 
	also yields an estimate on the BDDC preconditioner for the Virtual Element Method.
\end{remark}

%
%
%

%

\section{Numerical tests}\label{sec:numerical}
We consider the model problem%
\begin{equation}\label{eq:poisson}
  -\nabla\cdot(\rho \nabla u) = f\qquad \text{in }\Omega = (0,1)^3, \qquad
  u = 0\qquad \text{on }\partial\Omega. 
\end{equation}
In all experiments $\Omega$ is divided into $L = N\times N\times N$ cubic subdomains with side length $H = 1/N$, which are then each discretized by a rescaled version of a reference tessellation of the unit cube  made of either truncated octahedra or Voronoi cells (see Figure~\ref{fig:meshes}). In order to guarantee conformity, the mappings from the different subdomains to the reference unit cube are defined in such a way that each macro face $F$ is a symmetry plane for the tessellations on the two sides of $F$.  
Table~\ref{tab:meshdata} lists the values of the following geometrical parameters for the reference meshes used in the experiments:
\[
h = \max_{K\in\Th} h_K, \qquad h_\textup{min} = \min_{K\in\Th} h_{\textup{min},K}, \qquad  \gammastar = \min_{K\in\Th} \gammaK, 
\]
where $h_K$ is the diameter of element $K\in\Th$,
 $h_{\textup{min},K}$ is the minimum distance between any two vertices of $K$
 and $\gammaK$ is the parameter given in Definition ~\ref{shape_regular}.
From Table~\ref{tab:meshdata} we see that Voronoi meshes satisfy Assumption   \ref{ass:tessellation}  but with worse constants than the octahedra ones. 
For all tests the stabilization bilinear form $\Svem$ is the simplest one, defined in \eqref{stabform1}. 

\begin{figure}[htb]
\centering
\subfloat{\includegraphics[width=0.40\textwidth]{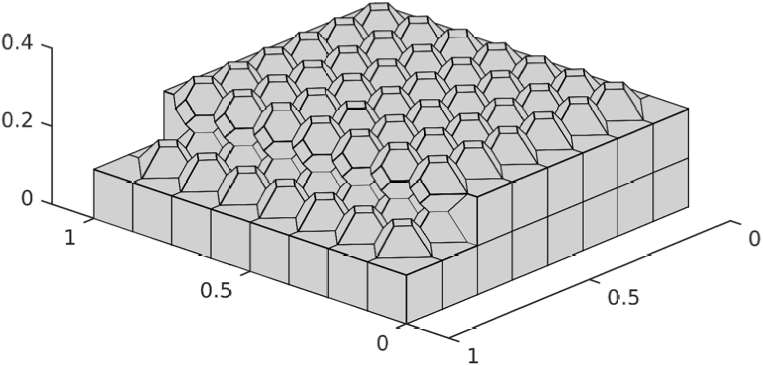}}\hspace{1cm}
\subfloat{\includegraphics[width=0.32\textwidth]{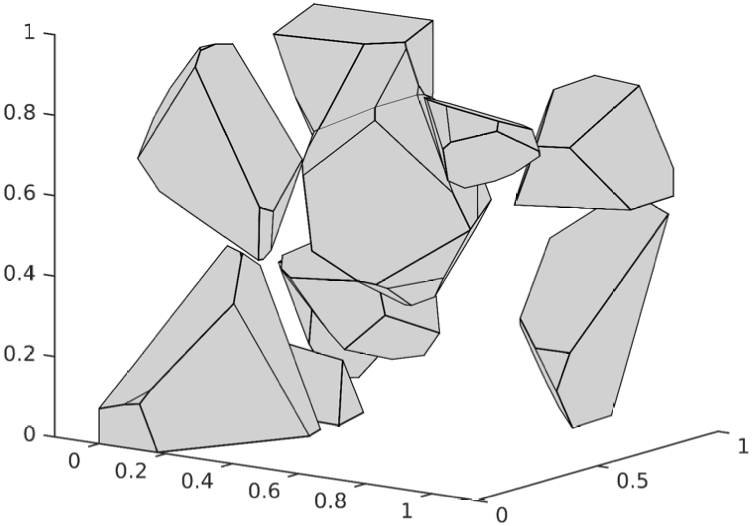}}
\caption{(Left) Clipped view of a mesh made of truncated octahedra; (right) example cells of a Voronoi mesh.}
\label{fig:meshes}
\end{figure}

\begin{table}[htb]
{
\small
  \caption{Data for the reference meshes used in the experiments: (left) meshes made of truncated octahedra; (right) meshes made of Voronoi cells.}
\label{tab:meshdata}
  \begin{center}
  \begin{tabular}{
      c
      S[table-format=1.{\roundPrecision}e-1]
      S[table-format=1.{\roundPrecision}e-1]
      S[table-format=1.{\roundPrecision}e-1]
|
      c
      S[table-format=1.{\roundPrecision}e-1]
      S[table-format=1.{\roundPrecision}e-1]
      S[table-format=1.{\roundPrecision}e-2]
    }
    \toprule
    Mesh & {$h$} & {$h_\textup{min}$} & {$\gamma_*$} &
    Mesh & {$h$} & {$h_\textup{min}$} & {$\gamma_*$}\\
    \midrule
	oct$_1$ & 4.330127e-01 & 6.250000e-02 & 6.060606e-02 & voro$_1$ & 7.159673e-01 & 2.403890e-04 & 9.079474e-07\\
    oct$_2$ & 2.886757e-01 & 4.166670e-02 & 6.060531e-02 & voro$_2$ & 5.656539e-01 & 2.479537e-04 & 4.833679e-07\\
    oct$_3$ & 2.165064e-01 & 3.125000e-02 & 6.060606e-02 & voro$_3$ & 5.369586e-01 & 1.258099e-04 & 3.476163e-08\\
    oct$_4$ & 1.732051e-01 & 2.500000e-02 & 6.060606e-02 & voro$_4$ & 3.616750e-01 & 2.602670e-05 & 6.765579e-09\\
    oct$_5$ & 1.443379e-01 & 2.083300e-02 & 6.060319e-02 & voro$_5$ & 2.806446e-01 & 5.688789e-06 & 1.916817e-08\\
    oct$_6$ & 1.237187e-01 & 1.785700e-02 & 6.060439e-02 & voro$_6$ & 2.456026e-01 & 3.116905e-06 & 3.383481e-09\\
    oct$_7$ & 1.082532e-01 & 1.562500e-02 & 6.060606e-02 & voro$_7$ & 1.896539e-01 & 1.313472e-06 & 7.239983e-11\\
    oct$_8$ & 9.622581e-02 & 1.388890e-02 & 6.060381e-02 & voro$_8$ & 1.477511e-01 & 5.927527e-08 & 5.317264e-12\\
    \bottomrule
  \end{tabular}
\end{center}}
\end{table}

\

Concerning the space $\widetilde W_h$, the following choices from Section~\ref{sec:FETI_prec} are tested:

\medskip

\begin{center}
{\sc E:} $\widetilde W_h =\rosa{\widetilde W_h^{(E)}},$ \hspace{.4cm}   {\sc F:} $\widetilde W_h =\rosa{\widetilde W_h^{(F)}},$ \hspace{.4cm}
	{\sc VE:} $\widetilde W_h = \rosa{\widetilde W_h^{(V)}} \cap\rosa{\widetilde W_h^{(E)}},$ \hspace{.4cm} {\sc VF:} $\widetilde W_h = \rosa{\widetilde W_h^{(V)}} \cap\rosa{\widetilde W_h^{(F)}}$.
\end{center}
	%

%
%

\medskip

\noindent For the sake of completeness, we also test   {\sc V:} $\widetilde W_h = \rosa{\widetilde W_h^{(V)}}$, 
 for which we expect a worse behaviour, as for the finite element case, where the condition number increases as \linebreak $(1+\log(H/h))^2(H/h)$ (see \cite{Toselli.Widlund},  Remark 6.39). 
 
 \
 
In order to analyze the performance of these algorithms, we carry out two series of experiments:

\begin{description}
	\item[Test 1:] {\em FETI-DP scalability.} 
		We increase the number of subdomains and the overall problem size, while keeping the reference mesh (and, consequently, the ratio $H/h$), fixed. Table~\ref{tab:test1-mesh} shows the dimension of the primal spaces $\widetilde W_h$ for the different algorithms. 
\noindent According to Theorem~\ref{bound_FETI-DP}, we expect the condition number for the FETI-DP preconditioner to remain asymptotically constant.\label{test1}

\item[Test 2:]  {\em FETI-DP quasi-optimality}. We fix the number of subdomains $L$, so that $H$ is kept constant, and increase the size of the local problems by choosing finer and finer reference meshes (from oct$_1$ to oct$_8$, or from voro$_1$ to voro$_8$ in Table~\ref{tab:meshdata}), thereby incrementing the overall problem size. This results in a decreasing $h$ and, consequently, increasing ratio $H/h$. 
\noindent According to Theorem~\ref{bound_FETI-DP}, we now expect the condition number for the FETI-DP preconditioner to asymptotically exhibit a polylogarithmic behavior.\label{test2}
\end{description}
%
%
%
%
%
To test the robustness of FETI-DP, each test is run with two types of data (unless otherwise stated):
\begin{itemize}
\item $\rho= 1$ in $\Omega$ and $f = \sin(2\pi x)\sin(2\pi y)\sin(2\pi z)$;
\item $\rho$ constant in each subdomain, assuming values  $\rho_1 = 10^5$ and $\rho_2 = 10^{-5}$ in a (3D) checkerboard distribution, so that there is a
coefficient jump of $10^{10}$ along $\Gamma$. The right hand side  $f$ is implicitly chosen by choosing a right hand side vector with values uniformly randomly distributed in $[-1,1]$.
\end{itemize}
Unless otherwise stated, we use the conjugate gradient with the FETI-DP Dirichlet preconditioner with zero initial guess and, as a stopping criterion, the relative reduction of the dual residual by either $10^{-6}$ or $10^{-12}$ when using the first or the second type of data, respectively. MATLAB$^\text{\textregistered}$ R2016b is used as the subdomain and coarse sparse direct solver. If not otherwise specified, results are obtained on a machine equipped with a processor Intel$^\text{\textregistered}$ Core$^\text{\texttrademark}$ i7-7820HQ, operating system Ubuntu Linux 16.04 LTS, memory 64 GB, 2400 MHz DDR4 Non-ECC SDRAM.

\begin{table}[thtb]
{
\small
\caption{Number of global and primal d.o.f.s for the different algorithms considered and the meshes of Test 1: 
	constant $H/h$, reference subdomain mesh oct$_3$ (octahedra) or vor$_5$ (Voronoi). 
}
\label{tab:test1-mesh}
	\begin{center}
	\begin{tabular}{
			S[table-format=4.0]
			S[table-format=1.{\roundPrecision}e+1]
			S[table-format=8.0]
			S[table-format=1.{\roundPrecision}e+1]
			S[table-format=8.0]
			S[table-format=4.0]
			S[table-format=4.0]
			S[table-format=4.0]
			S[table-format=4.0]
			S[table-format=4.0]
		}
			\toprule
			\multicolumn{1}{c}{\multirow{2}{*}{$L$}} 
			 & \multicolumn{2}{c}{{Octahedra}} & \multicolumn{2}{c}{{Voronoi}} & \multicolumn{5}{c}{{Coarse}} \\
			\cmidrule(lr){2-3} 	\cmidrule(lr){4-5} \cmidrule(lr){6-10} 
		         &    {$h$}        &  {d.o.f.}   &    {$h$} &  {d.o.f.} & {V} &  {E}  &  {F}  &  {VE}  & {VF} \\
			\midrule
		   8 & 1.0825e-01 &   31928 & 1.4032e-01 &    46416 &    1 &    6 &   12 &    7 &   13  \\
		  64 & 5.4127e-02 &  261960 & 7.0161e-02 &   388280 &   27 &  108 &  144 &  135 &  171 \\
		 216 & 3.6084e-02 &  892072 & 4.6774e-02 &  1330240 &  125 &  450 &  540 &  575 &  665 \\
		 512 & 2.7063e-02 & 2124248 & 3.5081e-02 &  3176952 &  343 & 1176 & 1344 & 1519 & 1687 \\
		1000 & 2.1651e-02 & 4160472 & 2.8064e-02 &  6233072 &  729 & 2430 & 2700 & 3159 & 3429 \\
		1728 & 1.8042e-02 & 7202728 & 2.3387e-02 & 10803256 & 1331 & 4356 & 4752 & 5687 & 6083 \\
		\bottomrule
	\end{tabular}
	\end{center}
	}
\end{table}

\subsection{FETI-DP scalability}



\begin{table}[htb]
{
\small
	\caption{Test 1. Truncated octahedra  meshes, $H/h = 4.6188$, and  constant coefficient $\rho=1$. Condition number estimates and iteration numbers for the different choices of $\widetilde W_h$.}
	\label{tab:test1-PiV-oct}
	\begin{center}
	\begin{tabular}{
			S[table-format=4.0]
			|
			S[table-format=2.{\roundPrecision}]
			S[table-format=2.0]
			|
			S[table-format=1.{\roundPrecision}]
			S[table-format=2.0]
			|
			S[table-format=1.{\roundPrecision}]
			S[table-format=2.0]
			|
			S[table-format=1.{\roundPrecision}]
			S[table-format=1.0]
			|
			S[table-format=1.{\roundPrecision}]
			S[table-format=2.0]
		}
		\toprule
			\multicolumn{1}{c}{\multirow{2}{*}{$L$}}  &	\multicolumn{2}{c}{  V} & \multicolumn{2}{c}{ E}  & \multicolumn{2}{c}{ F}  &  \multicolumn{2}{c}{VE}  & \multicolumn{2}{c}{VF}\\
			\cmidrule(lr){2-3} \cmidrule(lr){4-5}  	\cmidrule(lr){6-7} \cmidrule(lr){8-9} 	\cmidrule(lr){10-11} 
		 & 	{$\kappa$} & {it} & {$\kappa$} & {it} & {$\kappa$} & {it}& {$\kappa$} & {it}& {$\kappa$} & {it} \\
		\midrule
8   &   1.000000   &   1   &   1.000000   &   1   &   1.000000   &   1   &   1.000000   &   1   &   1.000000   &   1\\
64   &   3.234423   &   9   &   2.389406   &   8   &   2.331676   &   8   &   1.575555   &   7   &   2.150072   &   8\\
216   &   12.849566   &   12   &   2.671940   &   10   &   2.949196   &   10   &   1.584984   &   7   &   2.636949   &   10\\
512   &   13.748798   &   18   &   2.585374   &   10   &   3.046637   &   11   &   1.609582   &   7   &   2.567055   &   10\\
1000   &   13.923685   &   18   &   2.534773   &   10   &   3.270533   &   11   &   1.621458   &   7   &   2.702339   &   10\\
1728   &   13.981769   &   18   &   2.506587   &   10   &   3.382713   &   11   &   1.628285   &   7   &   2.748666   &   10\\
		\bottomrule
	\end{tabular}
\end{center}
}
\end{table}

\begin{table}[htb]
{
\small
	\caption{Test 1. Voronoi meshes, $H/h = 6.171687$,  and  constant coefficient $\rho=1$. Condition number estimates and iteration numbers for the different choices of $\widetilde W_h$.
		} 
\label{tab:test1-PiV-vor}
	\begin{center}
	\begin{tabular}{
				S[table-format=4.0]
				|
			S[table-format=2.{\roundPrecision}]
			S[table-format=2.0]
			|
			S[table-format=1.{\roundPrecision}]
			S[table-format=2.0]
			|
			S[table-format=2.{\roundPrecision}]
			S[table-format=2.0]
			|
			S[table-format=1.{\roundPrecision}]
			S[table-format=2.0]
			|
			S[table-format=1.{\roundPrecision}]
			S[table-format=2.0]
		}
		\toprule
			\multicolumn{1}{c}{\multirow{2}{*}{$L$}}  &	\multicolumn{2}{c}{V} & \multicolumn{2}{c}{ E}  & \multicolumn{2}{c}{ F}  &  \multicolumn{2}{c}{VE}  & \multicolumn{2}{c}{VF}\\
		\cmidrule(lr){2-3} \cmidrule(lr){4-5}  	\cmidrule(lr){6-7} \cmidrule(lr){8-9} 	\cmidrule(lr){10-11} 
		& {$\kappa$} & {it} & {$\kappa$} & {it} & {$\kappa$} & {it}& {$\kappa$} & {it}& {$\kappa$} & {it} \\
		\midrule
8   &   1.000000   &   1   &   1.000000   &   1   &   1.000000   &   1   &   1.000000   &   1   &   1.000000   &   1\\
64   &   27.881667   &   19   &   4.715470   &   15   &   6.363019   &   15   &   3.923381   &   13   &   5.506362   &   14\\
216   &   26.526605   &   25   &   6.032554   &   17   &   9.448031   &   19   &   4.158831   &   14   &   7.788726   &   17\\
512   &   29.047841   &   31   &   6.092926   &   17   &   10.537346   &   21   &   4.159487   &   14   &   8.375766   &   18\\
1000   &   29.551405   &   32   &   6.138289   &   17   &   11.030450   &   21   &   4.159843   &   14   &   8.731337   &   19\\
1728   &   30.473817   &   32   &   6.174897   &   17   &   11.433136   &   22   &   4.161037   &   14   &   8.890039   &   19\\
		\bottomrule
	\end{tabular}
	\end{center}
	}
	\end{table}


We test the scalability for both the constant coefficient case and the varying coefficients case on the lowest order method ($\p = 1$). We fix the number of subdomains to be $L = 216$ ($6 \times 6 \times 6$).
Results for the first series of experiments (Test 1) with constant
 coefficients $\rho$ 
are reported in Table~\ref{tab:test1-PiV-oct} 
for meshes of truncated octahedra, and in Table~\ref{tab:test1-PiV-vor} 
for Voronoi meshes. 
The results are in accordance with the theoretical bounds  for both sets of meshes. 
Tables \ref{tab:test1-PiV-oct} and \ref{tab:test1-PiV-vor} 
	show that, when no jumps in the coefficients are present, the results for {\sc E, F, VE},  and {\sc VF} are similar. 
In switching from the octahedra to the Voronoi mesh, which satisfies Assumption \ref{ass:tessellation} but with worse constants, we observe an increase in the number of iterations and in the condition number which, however, still displays the expected behavior. 
As one expects, the choice {\sc V} is, instead,  not competitive.

The numerical tests  with varying coefficients $\rho$  are displayed in  Tables~\ref{tab:test1-PiV-oct-rho} and \ref{tab:test1-PiV-vor-rho}
for octahedra and Voronoi meshes,  respectively.
With highly oscillating coefficients, when only face averages are selected as primal degrees of freedom (Algorithm {\sc F}), the preconditioner performs very poorly on both types of meshes.

\begin{table}[thtb]
{
\small
	\caption{
	Test 1. Truncated octahedra meshes,  $H/h$ constant, checkerboard $(\rho_1 = 10^5, \rho_2 = 10^{-5})$.
	Condition number estimates and iteration numbers for the different choices of $\widetilde W_h$.
 }
\label{tab:test1-PiV-oct-rho}
	\begin{center}
	\begin{tabular}{
				S[table-format=4.0]
				|
			S[table-format=2.{\roundPrecision}]
			S[table-format=2.0]
			|
			S[table-format=1.{\roundPrecision}]
			S[table-format=1.0]
			|
			S[table-format=1.{\roundPrecision}e+1]
			S[table-format=3.0]
			|
			S[table-format=1.{\roundPrecision}]
			S[table-format=1.0]
			|
			S[table-format=2.{\roundPrecision}]
			S[table-format=2.0]
		}
		\toprule
			\multicolumn{1}{c}{\multirow{2}{*}{$L$}} &	\multicolumn{2}{c}{V} & \multicolumn{2}{c}{ E}  & \multicolumn{2}{c}{ F}  &  \multicolumn{2}{c}{VE}  & \multicolumn{2}{c}{VF}\\
		\cmidrule(lr){2-3} \cmidrule(lr){4-5}  	\cmidrule(lr){6-7} \cmidrule(lr){8-9} 	\cmidrule(lr){10-11} 
		& {$\kappa$} & {it} & {$\kappa$} & {it} & {$\kappa$} & {it}& {$\kappa$} & {it}& {$\kappa$} & {it} \\
8   &   1.868217   &   6   &   1.745221   &   5   &   2.733039e+00   &   6   &   1.247625   &   5   &   2.168346   &   6\\
64   &   10.332885   &   13   &   2.572450   &   8   &   4.482920e+09   &   31   &   1.278404   &   5   &   11.691287   &   11\\
216   &   10.357991   &   15   &   2.528796   &   8   &   5.586315e+09   &   88   &   1.282275   &   5   &   12.615278   &   15\\
512   &   10.348820   &   15   &   2.540652   &   8   &   6.061157e+09   &   118   &   1.288959   &   5   &   13.373846   &   16\\
1000   &   10.360149   &   15   &   2.539138   &   8   &   6.294364e+09   &   132   &   1.288889   &   5   &   13.280532   &   16\\
1728   &   10.367114   &   15   &   2.544949   &   8   &   6.438213e+09   &   138   &   1.292265   &   5   &   13.168730   &   16\\
\bottomrule
	\end{tabular}
	\end{center}
	}
\end{table}

\begin{table}[thtb]
{\small
	\caption{Test 1. Voronoi meshes,  $H/h$ constant, checkerboard $(\rho_1 = 10^5, \rho_2 = 10^{-5})$. 
	Condition number estimates and iteration numbers for the different choices of $\widetilde W_h$.
}
\label{tab:test1-PiV-vor-rho}
	\begin{center}
	\begin{tabular}{
				S[table-format=4.0]
				|
			S[table-format=2.{\roundPrecision}]
			S[table-format=2.0]
			|
			S[table-format=1.{\roundPrecision}]
			S[table-format=2.0]
			|
			S[table-format=1.{\roundPrecision}e+1]
			S[table-format=3.0]
			|
			S[table-format=1.{\roundPrecision}]
			S[table-format=1.0]
			|
			S[table-format=2.{\roundPrecision}]
			S[table-format=2.0]
		}
		\toprule
			\multicolumn{1}{c}{\multirow{2}{*}{$L$}} &	\multicolumn{2}{c}{V} & \multicolumn{2}{c}{ E}  & \multicolumn{2}{c}{ F}  &  \multicolumn{2}{c}{VE}  & \multicolumn{2}{c}{VF}\\
		\cmidrule(lr){2-3} \cmidrule(lr){4-5}  	\cmidrule(lr){6-7} \cmidrule(lr){8-9} 	\cmidrule(lr){10-11} 
		& {$\kappa$} & {it} & {$\kappa$} & {it} & {$\kappa$} & {it}& {$\kappa$} & {it}& {$\kappa$} & {it} \\
		\midrule
8   &   4.950380   &   11   &   3.597147   &   9   &   6.742533e+00   &   11   &   3.197408   &   8   &   5.707291   &   11\\
64   &   23.731815   &   20   &   6.578026   &   13   &   1.136903e+10   &   71   &   3.551812   &   9   &   32.300596   &   19\\
216   &   25.534518   &   26   &   6.694528   &   13   &   1.450543e+10   &   175   &   3.583376   &   9   &   35.944505   &   27\\
512   &   25.730977   &   27   &   6.702935   &   13   &   1.573321e+10   &   228   &   3.581468   &   9   &   36.966806   &   28\\
1000   &   25.873318   &   27   &   6.682435   &   13   &   1.651172e+10   &   254   &   3.574793   &   9   &   37.812112   &   29\\
1728   &   25.914747   &   27   &   6.688536   &   13   &   1.694812e+10   &   269   &   3.577582   &   9   &   38.107560   &   29\\
		\bottomrule
	\end{tabular}
	\end{center}
	}
\end{table}

This is not in contradiction with our theoretical results. Indeed, with coefficient jumps of $10^{10}$ in a checkerboard distribution, we have $\CjumpE = 1, \CjumpF = \num[retain-unity-mantissa=false]{1e10}$, and $\CjumpEF = \num[retain-unity-mantissa=false]{1e10}$, in agreement with the better performance of {\sc E} with respect to {\sc F} and {\sc VF}, as shown in Tables~\ref{tab:test1-PiV-oct-rho},~\ref{tab:test1-PiV-vor-rho}.
The independence from the jumps of the coefficients is shown by {\sc VE} as predicted by the bound of Theorem~\ref{bound_FETI-DP}.

\subsection{FETI-DP quasi-optimality} We test  quasi optimality for both the constant coefficient case and the varying coefficients case, again on the lowest order method ($\p = 1$).
Results for our second set of runs (Test 2) with both smooth and random data are shown in Figures~\ref{fig:test2-oct}
for meshes of truncated octahedra, and in Figures~\ref{fig:test2-randvor-1-u} 
for Voronoi meshes. In both cases, the results are in agreement with the  estimates of Theorem~\ref{bound_FETI-DP}. In particular, the experiments show that FETI-DP achieves quasi-optimality for our model problem if {\sc VE} is used, i.e. $\widetilde W_h \subset \rosa{\widetilde W_h^{(V)}} \cap\rosa{\widetilde W_h^{(E)}}$. 

\begin{figure}[htb]
	\centering
	\begin{tabular}{rl}
		\begin{tikzpicture}[trim axis right]
		\begin{semilogxaxis}[legend style={font=\scriptsize}, small,  xlabel=Dof,ylabel=$\kappa^{1/2}$,width=0.40\textwidth,legend pos=north west ]
				\addplot [thick, color=orange, mark=o, solid] coordinates {
(74440, 2.15793)
(328600, 2.90632)
(892072, 3.58463)
(1889272, 4.25349)
(3444616, 4.98508)
(5682520, 5.61673)
(8727400, 6.34446)
(12703672, 6.76584)
		};
		\addlegendentry{V}
		\addplot [thick, color=red, mark=square, solid] coordinates {
(74440, 1.47468)
(328600, 1.58152)
(892072, 1.63461)
(1889272, 1.66418)
(3444616, 1.69988)
(5682520, 1.7158)
(8727400, 1.72066)
(12703672, 1.74449)
		};
		\addlegendentry{E}
		\addplot [thick, color=blue, mark=+, solid] coordinates {
(74440, 1.96196)
(328600, 1.80654)
(892072, 1.71732)
(1889272, 1.66845)
(3444616, 1.75465)
(5682520, 1.80283)
(8727400, 1.84309)
(12703672, 1.89678)
		};
		\addlegendentry{F}	
		\addplot [thick, color=green, mark=x, solid] coordinates {
(74440, 1.04676)
(328600, 1.17493)
(892072, 1.25896)
(1889272, 1.33114)
(3444616, 1.38431)
(5682520, 1.43929)
(8727400, 1.48232)
(12703672, 1.52612)
		};
		\addlegendentry{VE}
		\addplot [thick, color=magenta, mark=*, solid] coordinates {
(74440, 1.74771)
(328600, 1.63181)
(892072, 1.62387)
(1889272, 1.67918)
(3444616, 1.69053)
(5682520, 1.70656)
(8727400, 1.7796)
(12703672, 1.80998)
		};
		\addlegendentry{VF}
		\end{semilogxaxis}
		\end{tikzpicture}\ \hspace{1.cm} \
		\begin{tikzpicture}[trim axis right]
		\begin{semilogxaxis}[legend style={font=\scriptsize}, small,   xlabel=Dof,ylabel=$\kappa^{1/2}$,width=0.40\textwidth,legend pos=north west ]
\addplot [thick, color=orange, mark=o, mark options=solid, dashed] coordinates {
(74440, 2.06356)
(328600, 2.69584)
(892072, 3.21838)
(1889272, 3.70415)
(3444616, 4.13003)
(5682520, 4.50954)
(8727400, 4.87324)
(12703672, 5.20244)
		};    
		\addlegendentry{V}
		\addplot [thick, color=red, mark=square,mark options=solid, dashed] coordinates {
(74440, 1.48468)
(328600, 1.52197)
(892072, 1.59022)
(1889272, 1.62828)
(3444616, 1.68509)
(5682520, 1.70372)
(8727400, 1.77384)
(12703672, 1.78336)
		};
		\addlegendentry{E}
		\addplot [thick, color=green, mark=x,mark options=solid, dashed] coordinates {
(74440, 1)
(328600, 1.10597)
(892072, 1.13238)
(1889272, 1.22909)
(3444616, 1.32178)
(5682520, 1.34928)
(8727400, 1.36862)
(12703672, 1.43607)
		};
		\addlegendentry{VE}
		\addplot [thick, color=magenta, mark=*, mark options=solid,dashed] coordinates {
(74440, 2.72668)
(328600, 3.06725)
(892072, 3.5518)
(1889272, 3.97952)
(3444616, 4.33268)
(5682520, 4.70153)
(8727400, 5.06872)
(12703672, 5.33545)
		};
		\addlegendentry{VF}
		\end{semilogxaxis}
		\end{tikzpicture}

	\end{tabular}
	\caption{Test 2 - Truncated octahedra meshes. 
	Constant coefficients $\rho=1$ (feft)  and  varying coefficients, checkerboard $(\rho_1 = 10^5, \rho_2 = 10^{-5})$
	(right). 
		Plots of $\kappa^{1/2}$ as a function of the global degrees of freedom 
		 for the different choices of $\widetilde W_h$.
	}
	\label{fig:test2-oct}
\end{figure}

\begin{figure}[htb]		
	\centering
	\begin{tabular}{rl}
		\begin{tikzpicture}[trim axis right]
		\begin{semilogxaxis}[legend style={font=\scriptsize}, small,  xlabel=Dof,ylabel=$\kappa^{1/2}$,width=0.40\textwidth,legend pos=north west ]
				\addplot [thick, color=orange, mark=o, solid] coordinates {
			(68896, 2.84303)
			(148672, 3.195)
			(311416, 3.47658)
			(647848, 4.50879)
			(1330240, 5.12941)
			(2751304, 6.3283)
			(5573872, 7.50318)
			(11318584, 8.85401)
		};
		\addlegendentry{V}
		\addplot [thick, color=red, mark=square, solid] coordinates {
			(68896, 1.94483)
			(148672, 2.34104)
			(311416, 2.22839)
			(647848, 2.334)
			(1330240, 2.45613)
			(2751304, 2.74014)
			(5573872, 2.71532)
			(11318584, 2.77114)
		};
		\addlegendentry{E}
		\addplot [thick, color=blue, mark=+, solid] coordinates {
			(68896, 3.1525)
			(148672, 2.95106)
			(311416, 2.85512)
			(647848, 3.5232)
			(1330240, 3.07376)
			(2751304, 3.71473)
			(5573872, 2.88096)
			(11318584, 2.89638)
		};
		\addlegendentry{F}	
		\addplot [thick, color=green, mark=x, solid] coordinates {
			(68896, 1.53686)
			(148672, 1.5524)
			(311416, 1.77533)
			(647848, 1.98572)
			(1330240, 2.03932)
			(2751304, 2.17868)
			(5573872, 2.29989)
			(11318584, 2.35659)
		};
		\addlegendentry{VE}
		\addplot [thick, color=magenta, mark=*, solid] coordinates {
			(68896, 2.61483)
			(148672, 2.40472)
			(311416, 2.31918)
			(647848, 3.18592)
			(1330240, 2.79083)
			(2751304, 3.33854)
			(5573872, 2.65749)
			(11318584, 2.6737)
		};
		\addlegendentry{VF}
		\end{semilogxaxis}
		\end{tikzpicture} \ \hspace{1.cm}  \
			\begin{tikzpicture}[trim axis right]
			\begin{semilogxaxis}[legend style={font=\scriptsize}, small, xlabel=Dof,ylabel=$\kappa^{1/2}$,width=0.40\textwidth,legend pos=north west ]
				\addplot [thick, color=orange, mark=o, mark options=solid, dashed] coordinates {
			(68896, 2.80349)
			(148672, 3.14953)
			(311416, 3.61914)
			(647848, 4.31615)
			(1330240, 5.03716)
			(2751304, 5.69247)
			(5573872, 6.59732)
			(11318584, 7.69389)
		};    
		\addlegendentry{V}
		\addplot [thick, color=red, mark=square,mark options=solid, dashed] coordinates {
			(68896, 1.92252)
			(148672, 2.21223)
			(311416, 2.03273)
			(647848, 2.24789)
			(1330240, 2.58314)
			(2751304, 2.69077)
			(5573872, 2.64118)
			(11318584, 2.66362)
		};
		\addlegendentry{E}
		\addplot [thick, color=green, mark=x,mark options=solid, dashed] coordinates {
			(68896, 1.17765)
			(148672, 1.33033)
			(311416, 1.38763)
			(647848, 1.62744)
			(1330240, 1.89039)
			(2751304, 1.98277)
			(5573872, 2.12839)
			(11318584, 2.28268)
		};
		\addlegendentry{VE}
		\addplot [thick, color=magenta, mark=*, mark options=solid,dashed] coordinates {
			(68896, 4.07515)
			(148672, 4.07661)
			(311416, 4.50726)
			(647848, 6.17544)
			(1330240, 5.99991)
			(2751304, 8.1438)
			(5573872, 7.16402)
			(11318584, 8.07248)
		};
		\addlegendentry{VF}
		\end{semilogxaxis}
		\end{tikzpicture}
	\end{tabular}
	\caption{Test 2 - Voronoi meshes. Constant coefficients $\rho=1$ (left) and checkerboard $(\rho_1 = 10^5, \rho_2 = 10^{-5})$
	(right). 
		Plots of $\kappa^{1/2}$ as a function of the global degrees of freedom 
		for 
		different choices of $\widetilde W_h$.
	}
	\label{fig:test2-randvor-1-u}
\end{figure}

We conclude by a first test of the FETI-DP preconditioner on order $\p$ VEM, for $\p \not=1$ (more extensive testing will be carried out in the future). More precisely, for  $\p =1,2,3$ (we include the lowest order case for comparison), we test   the most efficient choice of  primal space, namely, Algorithm {\sc VE}, with, once again, two types of data:  constant coefficients $\rho=1$ 
and 
the 3D checkerboard coefficient distribution $(\rho_1 = 10^5, \rho_2 = 10^{-5})$. 
For this last test, the forcing term and Dirichlet boundary condition on $\partial\Omega$ are chosen so that 
the exact solution is  $u(x,y,z) = 1/(12\pi^2)\cos(2\pi x)\cos(2\pi y)\cos(2\pi z)$. The basis $\BasePoly_\p$ is chosen as the rescaled monomial basis and for the stabilization we use the diagonal recipe \eqref{stabform2}.

For this test the stopping criterion is the relative reduction of the dual residual by $10^{-6}$. In a C implementation, the PETSc built-in LU solver is used as subdomain solver, whereas we use GMRES with absolute convergence tolerance of $10^{-14}$ as the coarse solver. These last experiments are run on the EOS cluster of the University of Pavia (\url{http://matematica.unipv.it/it/cluster-eos}). We used nodes equipped with processors Intel(R) Xeon(R) Gold 6130 CPU @ 2.10 GHz, operating system GNU/Linux 3.10.0-957.1.3.el7.x86\_64.

Table~\ref{mesh_high_order} lists  the values of the geometrical parameters $h,h_{\min},  \gammastar$ for the reference meshes used in the experiments,  where Npol  denotes the number of polyhedra. Table~\ref{high_order_ro1_L64} shows, for the case of $\rho$ constant,  the number of iterations and the estimated condition number for the FETI-DP preconditioned system with a fixed number of subdomains ($L =  64$)
but varying both the polynomial degree $\p=1,2,3$ 
and  the reference mesh. 
The numerical tests with a 3D checkerboard 
distribution of $\rho$ are displayed in Table \ref{high_order_ro_L64}. 
Focusing, for $\p$ fixed, on the number of iterations rather that on the condition numbers (which are only estimated), we observe that the results are in agreement with the theory. As is to be expected, the performance deteriorates as $\p$ increases. Indeed, several of the inequalities on which the theoretical estimate relies on are (possibly heavily) depending on $\p$. In particular it is well known that, as $\p$ increases, the choice of the stabilization bilinear form $\Svem$ and of the face and interior degrees of freedom
become crucial for the performance of the method \cite{Mascotto_illcond,dassi_mascotto_3DVEM}. It is out of the scope of this paper to study the dependence of the performance of the preconditioner on such choices, and on the order $\p$ of the method, a question that we plan to address in the future.

\



%


\begin{table}
    \centering
       \caption{Reference meshes made of random Voronoi cells used in the experiments.}
       \label{mesh_high_order}
    \begin{tabular}{
    c
    S[table-format=3.0]
    S[table-format=1.{\roundPrecision}e-1]
    S[table-format=1.{\roundPrecision}e-1]
    S[table-format=1.{\roundPrecision}e-1]
    }
    \toprule
    Mesh & {Npol} & {$h$} & {$h_\textup{min}$} & {$\gamma_*$}\\
    \midrule
    voro-b$_1$ & 64 & 7.159673e-01 & 1.849483e-05 & 9.079474e-07\\
    voro-b$_2$ & 128 & 5.656539e-01 & 4.309501e-05 & 4.833679e-07\\
    voro-b$_3$ & 256 & 5.369586e-01 & 4.309501e-05 & 3.476163e-08\\
    voro-b$_4$ & 512 & 3.616750e-01 & 2.602670e-05 & 6.765579e-09\\
    \bottomrule
    \end{tabular}
\end{table}

\begin{table}
    \centering\small
 \caption{Number of subdomains $L = 4\times4\times4 = 64$, 
  diffusion coefficient $\rho = 1$.}
   \label{high_order_ro1_L64}
    \begin{tabular}{
    S[table-format=1.{\roundPrecision}e-1]
    S[table-format=6.0]
    S[table-format=1.{\roundPrecision}]
    S[table-format=2.0]
    S[table-format=6.0]
    S[table-format=2.{\roundPrecision}]
    S[table-format=2.0]
    S[table-format=7.0]
    S[table-format=2.{\roundPrecision}]
    S[table-format=2.0]
    }
    \toprule
    {\multirow{2}{*}{$h~\backslash~\p$}} & \multicolumn{3}{c}{$1$} & \multicolumn{3}{c}{$2$} & \multicolumn{3}{c}{$3$}  \\
    \cmidrule(lr){2-4} \cmidrule(lr){5-7} \cmidrule(lr){8-10}  
    & {d.o.f.} & {$\kappa$} & {it} & {d.o.f.} & {$\kappa$} & {it} & {d.o.f.} & {$\kappa$} & {it}  \\
    \midrule
    1.789918e-01 & 15587 & 2.299494538713966 & 8 & 76263 & 6.801359737275946 & 17 & 167675 & 12.45431045133344 & 24  \\
    1.414135e-01 & 36299 & 2.517613942822233 & 9 & 171063 & 7.386998568137607 & 19 & 371443 & 12.94023227678404 & 26  \\
    1.342397e-01 & 80651 & 2.491146489945803 & 10 & 369591 & 9.747558956441544 & 21 & 795443 & 21.72570456358937 & 30  \\
    9.041875e-02 & 172467 & 3.574232337317744 & 12 & 780135 & 15.12409611077458 & 26 & 1670939 & 26.53024056669756 & 34 \\
    \bottomrule
    \end{tabular}
 \end{table}

\begin{table}
    \centering\small
  \caption{Number of subdomains $L = 4\times4\times4 = 64$, 
  checkerboard $(\rho_1 = 10^5, \rho_2 = 10^{-5})$.}
   \label{high_order_ro_L64}
    \begin{tabular}{
    S[table-format=1.{\roundPrecision}e-1]
    S[table-format=6.0]
    S[table-format=1.{\roundPrecision}]
    S[table-format=2.0]
    S[table-format=6.0]
    S[table-format=2.{\roundPrecision}]
    S[table-format=2.0]
    S[table-format=7.0]
    S[table-format=2.{\roundPrecision}]
    S[table-format=2.0]
    S[table-format=7.0]
    S[table-format=2.{\roundPrecision}]
    S[table-format=2.0]
    }
    \toprule
    {\multirow{2}{*}{$h~\backslash~\p$}} & \multicolumn{3}{c}{$1$} & \multicolumn{3}{c}{$2$} & \multicolumn{3}{c}{$3$}  \\
    \cmidrule(lr){2-4} \cmidrule(lr){5-7} \cmidrule(lr){8-10} 
    & {d.o.f.} & {$\kappa$} & {it} & {d.o.f.} & {$\kappa$} & {it} & {d.o.f.} & {$\kappa$} & {it}  \\
    \midrule
    1.789918e-01 & 15587 & 1.741302383897261 & 8 & 76263 & 5.307941331651047 & 20 & 167675 & 10.34850678998972 & 26  \\
    1.414135e-01 & 36299 & 2.147411790774467 & 11 & 171063 & 7.302703845935320 & 23 & 371443 & 14.41474705258243 & 29  \\
    1.342397e-01 & 80651 & 2.368159745782593 & 11 & 369591 & 8.467694867298711 & 24 & 795443 & 15.62072206837484 & 30  \\
    9.041875e-02 & 172467 & 3.130723228654026 & 13 & 780135 & 8.198960414018742 & 26 & 1670939 & 13.15871299711487 & 30  \\
    \bottomrule
    \end{tabular}
   \end{table}

\bibliographystyle{siamplain}

\end{document}